\documentclass[12pt,a4paper,twoside]{article}

\title{\sc Poincar\'e meets Korn via Maxwell:\\
Extending Korn's First Inequality\\
to Incompatible Tensor Fields}
\def\shorttitle{Poincar\'e meets Korn via Maxwell}
\def\pauthor{Patrizio Neff, Dirk Pauly, Karl-Josef Witsch}

\def\mylabelonoff{off}
\def\allowdisbrk{no}

\usepackage{a4,exscale,ifthen,amsfonts,amssymb,amsmath,amscd,graphicx}
\usepackage[english]{babel}
\usepackage[mathscr]{eucal}

\author{{\sf\pauthor}}
\numberwithin{equation}{section}

\newenvironment{acknow}{{\vspace*{1cm}\noindent\bf Acknowledgements }}{}
\newcommand{\bewboxw}{\mbox{}\hfill $\square$ \\}

\newenvironment{proof}{{\noindent\bf Proof }}{\bewboxw}

\newcommand{\keywords}[1]{{\noindent\bf Key Words }#1}

\ifthenelse{\equal{\mylabelonoff}{on}}
{\newcommand{\mylabel}[1]{\label{#1}\fbox{{\rm #1}}}}{\newcommand{\mylabel}[1]{\label{#1}\makebox[0mm][]{}}}
\ifthenelse{\equal{\allowdisbrk}{yes}}
{\allowdisplaybreaks}{}

\usepackage[mathscr]{eucal}
\usepackage{color}
\usepackage[all]{xy}

\newcommand{\ds}{\displaystyle}
\newcommand{\ol}{\overline}
\newcommand{\ul}{\underline}

\newcommand{\nz}{\mathbb{N}}

\newcommand{\rz}{\mathbb{R}}

\newcommand{\rt}{\rz^3}
\newcommand{\rttt}{\rz^{3\times3}}

\newcommand{\rN}{\rz^N}

\DeclareMathOperator{\p}{\partial}

\newcommand{\na}{\nabla}

\DeclareMathOperator{\ed}{d}

\DeclareMathOperator{\cd}{\delta}

\DeclareMathOperator{\grad}{grad}
\DeclareMathOperator{\Grad}{Grad}

\DeclareMathOperator{\curl}{curl}
\DeclareMathOperator{\Curl}{Curl}

\renewcommand{\div}{\operatorname{div}}
\DeclareMathOperator{\Div}{Div}

\newcommand{\trans}[1]{{#1}^\top}
\newcommand{\T}{T}
\newcommand{\TS}{S}
\newcommand{\TR}{R}

\newcommand{\TP}{P}
\newcommand{\TF}{F}
\newcommand{\Tskew}{S}

\newcommand{\om}{\Omega}
\newcommand{\dom}{\p\!\om}
\newcommand{\omb}{\ol{\om}}
\newcommand{\ga}{\Gamma}
\newcommand{\gat}{\ga_{t}}
\newcommand{\gan}{\ga_{n}}
\newcommand{\gatj}{\ga_{t,j}}

\newcommand{\equi}{\Leftrightarrow}
\def\impl{\Rightarrow}

\newcommand{\qequi}{\quad\equi\quad}

\newcommand{\qimpl}{\quad\impl\quad}

\DeclareMathOperator{\supp}{supp}
\DeclareMathOperator{\sym}{sym}
\DeclareMathOperator{\dist}{dist}
\DeclareMathOperator{\interior}{int}
\DeclareMathOperator{\tr}{tr}
\renewcommand{\skew}{\operatorname{skew}}

\newcommand{\dvec}[3]{\begin{bmatrix}#1\\#2\\#3\end{bmatrix}}

\newcommand{\dmat}[9]{\begin{bmatrix}#1&#2&#3\\#4&#5&#6\\#7&#8&#9\end{bmatrix}}
\DeclareMathOperator{\id}{id}
\newcommand{\dl}{\,d\lambda}
\DeclareMathOperator{\SO}{SO}
\DeclareMathOperator{\so}{\mathfrak{so}}
\DeclareMathOperator{\dev}{dev}
\def\set#1#2{\{#1\,:\,#2\}}

\newcommand{\ch}{\hat{c}}
\newcommand{\ct}{\tilde{c}}

\newcommand{\cp}{c_{\mathtt{p}}}
\newcommand{\cpq}{c_{\mathtt{p},q}}

\newcommand{\cpz}{c_{\mathtt{p},0}}
\newcommand{\cpo}{c_{\mathtt{p},1}}
\newcommand{\cm}{c_{\mathtt{m}}}

\newcommand{\cpf}[1]{c_{\mathtt{pf},#1}}
\newcommand{\ck}{c_{\mathtt{k}}}
\newcommand{\ckF}{c_{\mathtt{k},\TF}}
\newcommand{\ckt}{c_{\mathtt{k},t}}
\newcommand{\cktF}{c_{\mathtt{k},t,\TF}}
\newcommand{\cks}{c_{\mathtt{k},s}}
\newcommand{\cksF}{c_{\mathtt{k},s,\TF}}
\newcommand{\cktj}{c_{\mathtt{k},t,j}}
\newcommand{\cksj}{c_{\mathtt{k},s,j}}

\newcommand{\chF}{\ch_{\TF}}
\newcommand{\cF}{c_{\TF}}

\DeclareMathOperator{\Lebesgue}{\mathsf{L}}
\newcommand{\Lgen}[2]{\Lebesgue^{#1}_{#2}}
\def\Li{\Lgen{\infty}{}}
\def\Liom{\Li(\om)}
\def\Lo{\Lgen{1}{}}
\def\Loom{\Lo(\om)}
\def\Lt{\Lgen{2}{}}
\def\Ltom{\Lt(\om)}

\def\Lq{\Lgen{q}{}}
\def\Lqom{\Lq(\om)}
\def\Lp{\Lgen{p}{}}
\def\Lpom{\Lp(\om)}

\newcommand{\qLt}[1]{\Lgen{2,#1}{}}
\newcommand{\Ltq}{\qLt{q}}
\newcommand{\Ltqpo}{\qLt{q+1}}
\newcommand{\Ltqmo}{\qLt{q-1}}
\newcommand{\Ltqom}{\Ltq(\om)}
\newcommand{\Ltqpoom}{\Ltqpo(\om)}
\newcommand{\Ltqmoom}{\Ltqmo(\om)}

\DeclareMathOperator{\DSobolev}{\mathsf{D}}
\newcommand{\Dgen}[3]{\overset{#3}{\DSobolev}{}^{#1}_{#2}}
\newcommand{\qD}[1]{\Dgen{#1}{}{}}

\newcommand{\qDc}[1]{\Dgen{#1}{}{\circ}}
\newcommand{\qDcz}[1]{\Dgen{#1}{0}{\circ}}
\newcommand{\Dq}{\qD{q}}

\newcommand{\Dqc}{\qDc{q}}

\newcommand{\Dqcz}{\qDcz{q}}

\newcommand{\Dqom}{\Dq(\om)}

\newcommand{\Dqcgatom}{\Dqc(\gat,\om)}
\newcommand{\Dqczgatom}{\Dqcz(\gat,\om)}

\DeclareMathOperator{\DeSobolev}{\Delta}
\newcommand{\Degen}[3]{\overset{#3}{\DeSobolev}{}^{#1}_{#2}}
\newcommand{\qDe}[1]{\Degen{#1}{}{}}
\newcommand{\qDec}[1]{\Degen{#1}{}{\circ}}

\newcommand{\qDecz}[1]{\Degen{#1}{0}{\circ}}
\newcommand{\Deq}{\qDe{q}}

\newcommand{\Deqc}{\qDec{q}}

\newcommand{\Deqcz}{\qDecz{q}}

\newcommand{\Deqom}{\Deq(\om)}

\newcommand{\Deqcganom}{\Deqc(\gan,\om)}

\newcommand{\Deqczganom}{\Deqcz(\gan,\om)}

\DeclareMathOperator{\Sobolev}{\mathsf{H}}
\newcommand{\Hgen}[3]{\overset{#3}{\Sobolev}{}^{#1}_{#2}}
\def\Ho{\Hgen{1}{}{}}

\def\Hoom{\Ho(\om)}

\def\Hoc{\Hgen{1}{}{\circ}}
\def\Hocom{\Hoc(\om)}
\def\Hocgatom{\Hoc(\gat;\om)}

\DeclareMathOperator{\WSobolev}{\mathsf{W}}
\newcommand{\Wgen}[3]{\overset{#3}{\WSobolev}{}^{#1}_{#2}}
\newcommand{\Wo}[1]{\Wgen{1,#1}{}{}}

\newcommand{\Woom}[1]{\Wo{#1}(\om)}

\DeclareMathOperator{\Cont}{\mathsf{C}}
\newcommand{\Cgen}[2]{\overset{#2}{\Cont}{}^{#1}}
\def\Cz{\Cgen{0}{}}

\def\Czomb{\Cz(\ol{\om})}
\def\Ci{\Cgen{\infty}{}}
\def\Cic{\Cgen{\infty}{\circ}}

\def\Co{\Cgen{1}{}}
\def\Coc{\Cgen{1}{\circ}}

\def\Cicom{\Cic(\om)}

\def\Cicgatom{\Cic(\gat,\om)}
\def\Cicganom{\Cic(\gan,\om)}

\def\Cocom{\Coc(\om)}

\DeclareMathOperator{\dirichlet}{\mathcal{H}}
\newcommand{\qharmdi}[2]{\dirichlet^{#1}_{#2}(\om)}
\newcommand{\harmdi}{\qharmdi{}{}}

\newcommand{\harmdiq}{\qharmdi{q}{}}

\newcommand{\Hggen}[3]{\overset{#2}{\Sobolev}(\grad;#3)}
\newcommand{\HGgen}[3]{\overset{#2}{\Sobolev}(\Grad;#3)}
\newcommand{\Hcgen}[3]{\overset{#2}{\Sobolev}(\curl_{#1};#3)}
\newcommand{\HCgen}[3]{\overset{#2}{\Sobolev}(\Curl_{#1};#3)}
\newcommand{\Hdgen}[3]{\overset{#2}{\Sobolev}(\div_{#1};#3)}
\newcommand{\HDgen}[3]{\overset{#2}{\Sobolev}(\Div_{#1};#3)}

\newcommand{\Hgom}{\Hggen{}{}{\om}}
\newcommand{\HGom}{\HGgen{}{}{\om}}
\newcommand{\Hcom}{\Hcgen{}{}{\om}}
\newcommand{\HCom}{\HCgen{}{}{\om}}
\newcommand{\Hdom}{\Hdgen{}{}{\om}}
\newcommand{\HDom}{\HDgen{}{}{\om}}
\newcommand{\Hgcom}{\Hggen{}{\circ}{\om}}

\newcommand{\Hccom}{\Hcgen{}{\circ}{\om}}
\newcommand{\HCcom}{\HCgen{}{\circ}{\om}}
\newcommand{\Hdcom}{\Hdgen{}{\circ}{\om}}

\newcommand{\Hczom}{\Hcgen{0}{}{\om}}
\newcommand{\HCzom}{\HCgen{0}{}{\om}}
\newcommand{\Hdzom}{\Hdgen{0}{}{\om}}
\newcommand{\HDzom}{\HDgen{0}{}{\om}}

\newcommand{\HDczom}{\HDgen{0}{\circ}{\om}}
\newcommand{\Hgcgatom}{\Hggen{}{\circ}{\gat,\om}}
\newcommand{\Hgcganom}{\Hggen{}{\circ}{\gan,\om}}
\newcommand{\HGcgatom}{\HGgen{}{\circ}{\gat,\om}}

\newcommand{\Hccgatom}{\Hcgen{}{\circ}{\gat,\om}}
\newcommand{\HCcgatom}{\HCgen{}{\circ}{\gat,\om}}
\newcommand{\Hdcgatom}{\Hdgen{}{\circ}{\gat,\om}}
\newcommand{\Hdcganom}{\Hdgen{}{\circ}{\gan,\om}}
\newcommand{\HDcganom}{\HDgen{}{\circ}{\gan,\om}}
\newcommand{\Hcczgatom}{\Hcgen{0}{\circ}{\gat,\om}}
\newcommand{\HCczgatom}{\HCgen{0}{\circ}{\gat,\om}}
\newcommand{\Hdczganom}{\Hdgen{0}{\circ}{\gan,\om}}
\newcommand{\HDczganom}{\HDgen{0}{\circ}{\gan,\om}}
\newcommand{\Hccganom}{\Hcgen{}{\circ}{\gan,\om}}
\newcommand{\HCcganom}{\HCgen{}{\circ}{\gan,\om}}

\DeclareMathOperator{\XSobolev}{\mathsf{X}}

\newcommand{\Xom}{\XSobolev(\om)}

\newcommand{\sfS}{\mathsf{S}}
\newcommand{\Som}{\sfS(\om)}
\newcommand{\Szom}{\sfS_{0}(\om)}
\newcommand{\RM}{\mathsf{RM}}
\newcommand{\BD}{\mathsf{BD}}
\newcommand{\BDom}{\BD(\om)}
\newcommand{\BV}{\mathsf{BV}}
\newcommand{\BVom}{\BV(\om)}

\newcommand{\normdst}{\hspace{-0.4ex}}
\newcommand{\scp}[2]{\left\langle#1,#2\right\rangle}
\newcommand{\scps}[2]{\langle#1,#2\rangle}
\newcommand{\scpom}[2]{\scp{#1}{#2}_{\om}}
\newcommand{\scpLtom}[2]{\scp{#1}{#2}_{\Ltom}}

\newcommand{\norm}[1]{\left|\normdst\left|#1\right|\normdst\right|}

\newcommand{\dnorm}[1]{\left|\normdst\left|\normdst\left|#1\right|\normdst\right|\normdst\right|}

\newcommand{\dnormF}[1]{\dnorm{#1}_{\TF}}

\newcommand{\normLtom}[1]{\norm{#1}_{\Ltom}}

\newcommand{\normLtqom}[1]{\norm{#1}_{\Ltqom}}

\newcommand{\normLtqpoom}[1]{\norm{#1}_{\Ltqpoom}}

\newcommand{\normLtqmoom}[1]{\norm{#1}_{\Ltqmoom}}

\newtheorem{lem}{Lemma}
\newtheorem{defi}[lem]{Definition}
\newtheorem{theo}[lem]{Theorem}
\newtheorem{cor}[lem]{Corollary}
\newtheorem{rem}[lem]{Remark}

\renewcommand{\TP}{\T}     
\renewcommand{\Tskew}{A} 
\newcommand{\dopo}{}

\begin{document}

\maketitle{}

\begin{center}
\tt Dedicated to Rolf Leis on the occasion of his 80th birthday
\end{center}

\vspace*{0mm}
\thispagestyle{empty}

\begin{abstract}
For a bounded domain $\om\subset\rt$ with Lipschitz boundary $\ga$
and some relatively open Lipschitz subset $\gat\neq\emptyset$ of $\ga$,
we prove the existence of some $c>0$, such that 
\begin{align}
\mylabel{mki}
c\norm{\TP}_{\Lt(\om,\rttt)}
\leq\norm{\sym\TP}_{\Lt(\om,\rttt)}
+\norm{\Curl\TP}_{\Lt(\om,\rttt)}
\end{align}
holds for all tensor fields in $\HCom$, i.e.,
for all square-integrable tensor fields $\TP:\om\to\rttt$
with square-integrable generalized rotation $\Curl\TP:\om\to\rttt$,
having vanishing restricted tangential trace on $\gat$.
If $\gat=\emptyset$, \eqref{mki} still holds at least for simply connected $\om$
and for all tensor fields $\TP\in\HCom$
which are $\Ltom$-perpendicular to $\so(3)$, i.e.,
to all skew-symmetric constant tensors.
Here, both operations, $\Curl$ and tangential trace,
are to be understood row-wise.

For compatible tensor fields $\TP=\na v$,
\eqref{mki} reduces to a non-standard variant
of the well known Korn's first inequality in $\rt$, namely
$$c\norm{\na v}_{\Lt(\om,\rttt)}
\leq\norm{\sym\na v}_{\Lt(\om,\rttt)}$$
for all vector fields $v\in\Ho(\om,\rt)$,
for which $\na v_{n}$, $n=1,\dots,3$, are normal at $\gat$.
On the other hand, identifying vector fields $v\in\Ho(\om,\rt)$ 
(having the proper boundary conditions)
with skew-symmetric tensor fields $\TP$,
\eqref{mki} turns to Poincar\'e's inequality since
$$\sqrt{2}c\norm{v}_{\Lt(\om,\rt)}
=c\norm{\TP}_{\Lt(\om,\rttt)}
\leq\norm{\Curl\TP}_{\Lt(\om,\rttt)}
\leq2\norm{\na v}_{\Lt(\om,\rt)}.$$
Therefore, \eqref{mki} may be viewed 
as a natural common generalization 
of Korn's first and Poincar\'e's inequality.
From another point of view, \eqref{mki} states
that one can omit compatibility of the tensor field $\TP$ at the expense 
of measuring the deviation from compatibility through $\Curl\TP$. 
Decisive tools for this unexpected estimate 
are the classical Korn's first inequality, 
Helmholtz decompositions for mixed boundary conditions 
and the Maxwell estimate.\\

\end{abstract}

\tableofcontents

\section{Introduction}

In this contribution we show that Korn's first inequality 
can be generalized in some not so obvious directions, 
namely to tensor fields which are not gradients. 
Our study is a continuation from
\cite{neffpaulywitschgenkornpamm,neffpaulywitschgenkornrt,
neffpaulywitschgenkornrtzap,neffpaulywitschgenkornrn}
and here we generalize our results to weaker boundary conditions 
and domains of more complicated topology.
For the proof of our main inequality \eqref{mki} 
we combine techniques from electro-magnetic and elasticity theory, namely
\begin{description}
\item[\rm\quad(HD)] Helmholtz' decomposition,
\item[\rm\quad(MI)] the Maxwell inequality,
\item[\rm\quad(KI)] Korn's inequality.
\end{description}
Since these three tools are crucial for our results
we briefly look at their history.
As pointed out in the overview \cite{sproessighelmdecooverview},
Helmholtz founded a comprehensive development
in the theory of projections methods 
mostly applied in, e.g.,
electromagnetic or elastic theory or
fluid dynamics. His famous theorem HD, see Lemma \ref{helmdeco}, states,
that any sufficiently smooth and sufficiently fast decaying vector field 
can be characterized by its rotation and divergence 
or can be decomposed into an irrotational
and a solenoidal part.
A first uniqueness result was given by Blumenthal 
in \cite{blumenthalhelmdecounique}.
Later, Hilbert and Banach space methods have been used
to prove similar and refined decompositions of the same type.

The use of inequalities is widespread in establishing existence and uniqueness
of solutions of partial differential equations. 
Furthermore, often these inequalities ensure 
that the solution is in a more suitable space 
from a numerical view point than the solution space itself.
Let $\om\subset\rt$ be a bounded domain with Lipschitz continuous boundary $\ga$.
Moreover, let $\gat,\gan$ be some relatively open Lipschitz subsets of $\ga$ 
with $\ol{\gat}\cup\ol{\gan}=\ga$ and $\gat \neq\emptyset$.
In potential theory use is made of {\bf Poincar\'e's inequality}, this is
\begin{align}
\mylabel{intropoincare}
\normLtom{u}&\leq\cp\normLtom{\na u}
\end{align}
for all functions $u\in\Hocgatom$\footnote{For 
exact definitions see section \ref{defsec}.} 
with some constant $\cp>0$\footnote{In the following 
$\cp,\ck,\cm>0$ refer to the constants
in Poincar\'e's, Korn's and in the Maxwell inequalities, respectively.},
to bound the scalar potential in terms of its gradient. 
In elasticity theory {\bf Korn's first inequality}
in combination with Poincar\'e's inequality, this is
\begin{align}
\mylabel{introkorn}
(\cp^2+1)^{-1/2}\norm{v}_{\Hoom}\leq\normLtom{\na v}&\leq\ck\normLtom{\sym\na v}
\end{align}
for all vector fields $v\in\Hocgatom$
with some constant $\ck>0$,
is needed for bounding the deformation 
of an elastic medium in terms of the symmetric strains, i.e., 
the symmetric part $\sym\na v=\frac{1}{2}(\na v+\trans{(\na v)})$ 
of the Jacobian $\na v$.
In electro-magnetic theory the
{\bf Maxwell inequality} (see Lemma \ref{poincaremax}), this is
\begin{align}
\mylabel{intromaxwell}
\normLtom{v}&\leq\cm\big(\normLtom{\curl v}+\normLtom{\div v}\big)
\end{align}
for all $v\in\Hccgatom\cap\Hdcganom\cap\harmdi^{\bot}$
with some positive $\cm$,
is used to bound the electric and magnetic field in terms of 
the electric charge and current density, respectively.
Actually, this important inequality is just the continuity estimate 
of the corresponding electro-magneto static solution operator. 
It has different names in the literature, e.g.,  
Friedrichs', Gaffney's or Poincar\'e 
type inequality \cite{Gaffney55,friedrichsdiffformsriemann}.

It is well known that Korn's and Poincar\'e's inequality are not equivalent. 
However, one main result of our paper is 
that both inequalities, i.e., \eqref{intropoincare}, \eqref{introkorn}, 
can be inferred from the more general result \eqref{mki},
where \eqref{intromaxwell} is used within the proof.

\subsection{The Maxwell inequality}

Concerning the MI (Lemma \ref{poincaremax})
in 1968 Leis \cite{leistheoem} 
considered the boundary value problem 
of total reflection for the inhomogeneous and anisotropic Maxwell system 
as well in bounded as in exterior domains. 
For bounded domains $\om\subset\rt$ he was able to estimate 
the derivatives of vector fields $v$ by the fields themselves,
their divergence and their rotation in $\Ltom$, i.e.,
\begin{align}
\mylabel{coe}
c\sum_{n=1}^3\normLtom{\p_{n}v}
&\leq\normLtom{v}+\normLtom{\curl v}+\normLtom{\div v},
\end{align}
provided that the boundary $\ga$ is sufficiently smooth 
and that $\nu\times v|_{\ga}=0$\footnote{$\nu$ denotes 
the outward unit normal at $\ga$ and $\times$ 
respectively $\cdot$ the vector respectively scalar product in $\rt$.},
i.e., the tangential trace of $v$ vanishes at $\ga$.
Of course, \eqref{coe} implies
\begin{align}
\nonumber
c\norm{v}_{\Hoom}
&\leq\normLtom{v}+\normLtom{\curl v}+\normLtom{\div v}
\end{align}
and thus by Rellich's selection theorem
the {\bf Maxwell compactness property} (MCP), i.e.,
\begin{align*}
\Xom&:=\Hccom\cap\Hdom\\
&\,\,=\set{v\in\Ltom}{\curl v\in\Ltom,\,\div v\in\Ltom,\,\nu\times v|_{\ga}=0}
\end{align*}
is compactly embedded into $\Ltom$, since $\Xom$ is a closed subspace 
of the Sobolev-Hilbert space $\Hoom$. 
However, \eqref{coe},
which is often called Friedrichs' or Gaffney's inequality, 
fails if smoothness of $\partial \Omega$ is not assumed.
On the other hand, by a standard indirect argument the MCP 
implies the Maxwell inequality
\eqref{intromaxwell} for $\gat=\ga$.
Hence, the compact embedding 
\begin{align}
\mylabel{compembeddintro}
\Xom\hookrightarrow\Ltom
\end{align}
is crucial for a solution theory suited for Maxwell's equations
as well as for the validity of the Maxwell estimate \eqref{intromaxwell} 
or Lemma \ref{poincaremax}.
But in the non-smooth case compactness of \eqref{compembeddintro} 
can not be proved by Rellich's selection theorem.
On the other hand, if \eqref{compembeddintro} is compact, 
one obtains Fredholm's alternative
for time-harmonic/static Maxwell equations
and the Maxwell inequality for bounded domains. 
For unbounded domains, e.g., exterior domains,
(local) compactness implies
Eidus' limiting absorption and limiting amplitude principles 
and the corresponding weighted Maxwell inequalities
\cite{eiduslabp,eiduslamp,eiduslamptwo,eidusmixproela}. 
These are the right and crucial tools for treating radiation problems,
see the papers by Pauly 
\cite{paulytimeharm,paulystatic,paulydeco,paulyasym,kuhnpaulyregmax}
for the latest results.
Therefore, Leis encouraged some of his students
to deal with electro-magnetic problems,
in particular with the MCP-question, see
\cite{picardpotential,picardboundaryelectro,picardcomimb,
picardlowfreqmax,picarddeco,picardweckwitschxmas,rinkensdiss,weckmax,witschremmax}. 

In 1969 Rinkens \cite{rinkensdiss} (see also \cite{leisbook}) 
presented an example of a non-smooth domain 
where the embedding of $\Xom$ into $\Hoom$ is not possible.  
Another example had been found shortly later 
and is written down in a paper by Saranen \cite{saranenmaxkegel}. 

Henceforth, there was a search for proofs
which do not make use of an embedding of $\Xom$ into $\Hoom$. 
In 1974 Weck \cite{weckmax} obtained a first and quite general result 
for `cone-like' regions. 
Weck considered a generalization of Maxwell's boundary value problem 
to Riemannian manifolds of arbitrary dimension $N$, 
going back to Weyl \cite{weylstrahlungsfelder}. 
The cone-like regions have Lipschitz boundaries 
but maybe not the other way round. 
However, polygonal boundaries are covered by Weck's result. 
In a joint paper by Picard, Weck and Witsch \cite{picardweckwitschxmas} 
Weck's proof has been modified 
to obtain \eqref{compembeddintro} even for domains 
which fail to have Lipschitz boundary.

Other proofs of \eqref{compembeddintro} for Lipschitz domains have been given by
Costabel \cite{costabelremmaxlip} and Weber \cite{webercompmax}. 
Costabel showed that $\Xom$ is already embedded 
into the fractional Sobolev space $\Hgen{1/2}{}{}(\om)$. 
Weber's proof has been modified by Witsch \cite{witschremmax}
to obtain the result for domains with H\"older continuous boundaries 
(with exponent $p>1/2$). 
Finally, there is a quite elegant result by Picard \cite{picardcomimb}
who showed that if the result holds for smooth boundaries
it holds for Lipschitz boundaries as well. 
This result remains true even 
in the generalized case (for Riemannian manifolds).

In this paper we shall make use of a result 
by Jochmann \cite{jochmanncompembmaxmixbc} 
who allows a Lipschitz boundary $\ga$ which is divided into two parts 
$\gat$ and $\gan$ by a Lipschitz curve 
and such that on $\gat$ and $\gan$ 
the mixed boundary conditions $\nu\times v|_{\gat}=0$ and 
$\nu\cdot v|_{\gan}=0$\, respectively, hold. 
In his dissertation, Kuhn \cite{kuhndiss} has proved 
an analogous result for the generalized Maxwell equations 
on Riemannian manifolds, following Weck's approach. 

The well known Sobolev type space $\Hcom$ has plenty of important and prominent applications,
most of them in the comprehensive theory of Maxwell's equations,
i.e., in electro-magnetic theory.
Among others, we want to mention
\cite{leistheoem,leiseindfortmax,leisbook,picardpotential,picardboundaryelectro,
picardcomimb,picardlowfreqmax,picarddeco,weckmax,wecktrace,witschremmax,webercompmax,
weberregmax,paulytimeharm,paulystatic,paulydeco,paulyasym,kuhnpaulyregmax,paulyrepinmaxst}. 
It is also used as a main tool for the analysis and discretization 
of Navier-Stokes' equations and in the numerical analysis 
of non-conforming finite element discretizations \cite{Hiptmair09,Raviart86}.

\subsection{Korn's inequality}

Korn's inequality gives the control of the 
$\Ltom$-norm of the gradient of a vector field 
by the $\Ltom$-norm of just the symmetric part of its gradient, 
under certain conditions. 
The most elementary variant of Korn's inequality for $\gat=\ga$ reads as follows: 
For any smooth vector field $v:\om\to\rt$ with compact support in $\om$,
i.e., $v\in\Cicom$,
\begin{align}
\mylabel{korntrivial}
\normLtom{\na v}^2&\leq2\normLtom{\sym\na v}^2
\intertext{holds. This inequality is simply obtained by straight forward partial integration,
see the Appendix, and dates back to Korn himself \cite{Korn09}.
Moreover, it can be improved easily by estimating just the 
deviatoric part of the symmetric gradient (see Appendix), this is}
\mylabel{korntrivialdev}
\forall\,v\in\Hocom\dopo\quad\frac{1}{2}\normLtom{\na v}^2
&\leq\normLtom{\dev\sym\na v}^2\leq\normLtom{\sym\na v}^2.
\end{align}
Here, we introduce the deviatoric part $\dev\T:=\T-\frac{1}{3}\tr\T\id$ 
as well as the symmetric and skew-symmetric parts
$\sym\T:=\frac{1}{2}(\T+\trans{\T})$, $\skew\T:=\frac{1}{2}(\T-\trans{\T})$
for quadratic matrix or tensor fields $\T$.
Note that $\T=\sym\T+\skew\T=\dev\T+\frac{1}{3}\tr\T\id$ 
and $\sym\T$, $\skew\T$ and $\dev\T$, $\tr\T\id$ are orthogonal in $\rttt$.
Together with (component-wise) 
Poincar\'e's inequality \eqref{intropoincare} for $\gat=\ga$,
see \cite{Poincare1894}, one arrives as in \eqref{introkorn} for all $v\in\Hocom$ at 
$$(\cp^2+1)^{-1/2}\norm{v}_{\Hoom}\leq\normLtom{\na v}
\leq\sqrt{2}\normLtom{\dev\sym\na v}\leq\sqrt{2}\normLtom{\sym\na v}.$$
Then, Rellich's selection theorem shows that the set of all 
$\Hocom$-vector fields whose (deviatoric) symmetric gradients 
are bounded in $\Ltom$ is (sequentially) compact in $\Ltom$.

Let us mention that Arthur Korn (1870-1945) was a student of Henri Poincar\'e. 
Korn visited him in Paris before the turn of the 20th century
and it was again Korn who wrote the obituary for Poincar\'e in 1912 \cite{Korn12}.
It is also worth mentioning that Poincar\'e helped 
to introduce Maxwell's electro-magnetic theory to french readers.
The interesting life of the german-jewish mathematician,
physicist and inventor of telegraphy Korn is recalled in
\cite{Korn_biography_german,Korn_biography_french}.

In general, Korn's inequality involves 
an integral measure of shape deformation, i.e.,
a measure of strain $\norm{\sym\na v}$, 
with which it is possible to control the distance of the deformation 
to some Euclidean motion or to control the $\Hoom$-norm or semi-norm. 

Consider the kernel of the linear operator $\sym\na:\Hoom\subset\Ltom\to\Ltom$
\begin{align}
\mylabel{rmkernsymgrad}
\ker(\sym\na)=\RM:=\set{x\mapsto\Tskew x+b}{\Tskew\in\so(3),\,b\in\rt},
\end{align}
the space of all infinitesimal rigid displacements (motions)
which consists of all affine linear transformations $v$ 
for which $\na v=\Tskew\in\so(3)$, i.e.,
$\Tskew$ is skew-symmetric\footnote{\eqref{rmkernsymgrad} easily 
follows from the simple observation that $\sym\na v=0$ 
implies $\na v(x)=\Tskew(x)\in\so(3)$.
Taking the $\Curl$ on both sides gives $\Curl\Tskew=0$ and thus $\na\Tskew=0$.
Hence, $\Tskew$ must be a constant skew-symmetric matrix.
Equivalently, one may use the well known representation for second derivatives
$\p_{i}\p_{j}v_{k}=\p_j(\sym\na v)_{ik}+\p_i(\sym\na v)_{jk}-\p_k(\sym\na v)_{ij}$.
Then, $\sym\na v=0$ implies that $v$ is a first order polynomial. 
This representation formula for second derivatives of $v$ 
in terms of derivatives of strain components 
can also serve as basis for a proof 
of Korn's second inequality \cite{Ciarlet10,Duvaut76}. 
In this case one uses the lemma of Lions, see \cite{Ciarlet98a}, i.e.,
for a Lipschitz domain $u\in\Ltom$ if and only if 
$u\in\Hgen{-1}{}{}(\om)$ and $\na u\in\Hgen{-1}{}{}(\om)$.}.
Since the measure of strain $\sym\na v$ is invariant 
with respect to superposed infinitesimal rigid displacements,
i.e., $\RM\subset\ker(\sym\na)$,
one needs some linear boundary or normalization conditions 
in order to fix this Euclidean motion. 
E.g., using homogeneous Dirichlet boundary conditions 
one has \eqref{introkorn}.
By normalization one gets
\begin{align}
\mylabel{kornwithoutbcone}
\normLtom{\na v}
&\leq\ck\normLtom{\sym\na v}
\end{align}
for all $v\in\Hoom$ 
with $\na v\bot\so(3)$\footnote{$\bot$ denotes 
orthogonality in $\Ltom$, whose elements map 
into $\rz$, $\rt$ or $\rttt$, respectively.}.
Equivalently, one has for all $v\in\Hoom$, e.g.,
\begin{align}
\mylabel{kornwithoutbctwo}
\normLtom{\na v-\Tskew_{\na v}}&\leq\ck\normLtom{\sym\na v},
\end{align}
where the constant skew-symmetric tensor
$$\Tskew_{\na v}:=\skew\oint_{\om}\na v\dl\in\so(3),\quad
\oint_{\om}u\dl:=\lambda(\om)^{-1}\int_{\om}u\dl\quad
\text{($\lambda$: Lebesgue's measure)},$$
is the $\Ltom$-orthogonal projection of $\na v$
onto $\so(3)$.
For details we refer to the appendix.
Poincar\'e's inequalities for vector fields by normalization read
\begin{align}
\mylabel{poincarewithoutbcone}
\normLtom{v}
&\leq\cp\normLtom{\na v},&
\norm{v}_{\Hoom}
&\leq(1+\cp^2)^{1/2}\normLtom{\na v}
\intertext{for all $v\in\Hoom$ with $v\bot\,\rt$.
Equivalently, one has for all $v\in\Hoom$, e.g.,}
\mylabel{poincarewithoutbctwo}
\normLtom{v-a_{v}}
&\leq\cp\normLtom{\na v},&
\norm{v-a_{v}}_{\Hoom}
&\leq(1+\cp^2)^{1/2}\normLtom{\na v},
\end{align}
where the constant vector
$$a_{v}:=\oint_{\om}v\dl\in\rt$$ 
is the $\Ltom$-orthogonal projection of $v$ onto $\rt$.
Combining \eqref{kornwithoutbcone} and \eqref{poincarewithoutbcone} we obtain
\begin{align}
\mylabel{kornwithpoincarewithoutbcone}
(1+\cp^2)^{-1/2}\norm{v}_{\Hoom}\leq\normLtom{\na v}&\leq\ck\normLtom{\sym\na v}
\intertext{for all $v\in\Hoom$ with $\na v\bot\so(3)$ and $v\bot\,\rt$.
Without these conditions one has}
\mylabel{kornwithpoincarewithoutbctwo}
(1+\cp^2)^{-1/2}\norm{v-r_{v}}_{\Hoom}
\leq\normLtom{\na v-\Tskew_{\na v}}
&\leq\ck\normLtom{\sym\na v},
\end{align}
for all $v\in\Hoom$, where the rigid motion 
$r_{v}:=\Tskew_{\na v}\xi+a_{v}-\Tskew_{\na v}a_{\xi}\in\RM$ 
with the identity function $\xi(x):=\id(x)=x$ reads
$$r_{v}(x):=\Tskew_{\na v}x+\oint_{\om}v\dl-\Tskew_{\na v}\oint_{\om}x\dl_{x}.$$
Note that $u:=v-r_{v}$ belongs to $\Hoom$ with $\na u=\na v-\Tskew_{\na v}$
and satisfies $\na u\bot\so(3)$ and $u\bot\,\rt$.
Hence \eqref{kornwithpoincarewithoutbcone} holds for $u$.
Moreover, we have for $v\in\Hoom$
$$r_{v}=0\qequi\Tskew_{\na v}=0\,\wedge\,a_{v}=0
\qequi\na v\bot\so(3)\,\wedge\,v\bot\,\rt.$$
See the appendix for details.
Conditions to eliminate some or all six rigid body modes 
(three infinitesimal rotations and three translations) 
comprise (see \cite{Alessandrini08})
$$\skew\int_{\om}\na v\dl=0,\quad v|_{\gat}=0,\quad
\na v_{n}\text{ normal to }\gat.$$

Korn's inequality is the main tool in showing existence, 
uniqueness and continuous dependence upon data 
in linearized elasticity theory 
and it has therefore plenty of applications 
in continuum mechanics \cite{Oleinik92,Horgan95}. 
One refers usually to \cite{Korn06,Korn08,Korn09} 
for first versions of Korn's inequalities. 
These original papers by Korn are, however,
difficult to read nowadays and Friedrichs even claims 
that they are wrong \cite{Friedrichs47}.
In any case, in \cite[p.710(13)]{Korn09} Korn states 
that (in modern notation)
$$\normLtom{\skew\na v}\leq\ck\normLtom{\sym\na v}$$
holds for all vector fields $v:\om\subset\rt\to\rt$ 
having H\"older continuous first order derivatives 
and which satisfy
$$\int_{\om}v\dl=0,\quad\int_{\om}\skew\na v\dl=0.$$
Note that this implies $\Tskew_{\na v}=0$ and $a_{v}=0$
and hence $r_{v}=0$.

Let $\gat\neq\emptyset$. By the classical {\bf Korn's first inequality}
with homogeneous Dirichlet boundary condition we mean
\begin{align*}
\exists\,\ck&>0&\forall\,v&\in\Hocgatom&
\normLtom{\na v}&\leq\ck\normLtom{\sym\na v}
\intertext{or equivalently by Poincar\'e's inequality \eqref{intropoincare}}
\exists\,\ck&>0&\forall\,v&\in\Hocgatom&
\norm{v}_{\Hoom}&\leq\ck\normLtom{\sym\na v},
\intertext{see \eqref{introkorn},
whereas we say that the classical {\bf Korn's second inequality} holds if}
\exists\,\ck&>0&\forall\,v&\in\Hoom&
\norm{v}_{\Hoom}&\leq\ck\big(\normLtom{v}+\normLtom{\sym\na v}\big).
\end{align*}

Korn's first inequality can be obtained as a consequence 
of Korn's second inequality\footnote{The ascription 
Korn's first or second inequality is not universal. 
Friedrichs \cite{Friedrichs47} refers to Korn's inequality 
$\normLtom{\na v}\leq\ck\normLtom{\sym\na v}$
in the first case if $u|_{\ga}=0$ 
and to the second case if $\skew\int_{\om}\na v\dl$ vanishes.
We follow the usage in \cite[p.54]{Valent88}.} 
and the compactness of the embedding $\Hoom\hookrightarrow\Ltom$, i.e.,
Rellich's selection theorem for $\Hoom$. 
Thus, the main task for Korn's inequalities is to show Korn's second inequality. 
Korn's second inequality in turn can be seen as a strengthened version 
of G\r{a}rding's inequality requiring methods 
from Fourier analysis \cite{Necas68a,Necas68b,Kato79}.
Very elegant and short proofs of Korn's second inequality 
have been presented in \cite{Oleinik89,Tiero99}
and by Fichera \cite{Fichera72}. 
Fichera's proof can be found in the appendix of Leis' book \cite{leisbook}.
Another short proof is based on strain preserving extension operators \cite{Nitsche81}.

Both inequalities admit a natural extension to the Sobolev space 
$\Woom{p}$ for Sobolev exponents $1<p<\infty$.
The first proofs have been given by Mosolov and Mjasnikov 
in \cite{Mosolov71,Mosolov72} and by Ting in \cite{Ting72}.
Note that Korn's inequalities are wrong\footnote{Korn's inequalities 
are also wrong in $\Woom{\infty}$. 
E.g., consider the unit ball in $\rz^2$ 
and the vector field $v(x):=\ln|x|(x_{2},-x_{1})$.}
in $\Woom{1}$, see \cite{Ornstein62}. 
New and simple counterexamples for $\Woom{1}$ have been obtained in \cite{Conti05a}.
Friedrichs furnished the first\footnote{The case $N=2$ 
has already been proved by Friedrichs \cite{Friedrichs37} in 1937.} 
modern proof of the above inequalities \cite{Friedrichs47}, see also
\cite{Payne61,Gobert62,Friedrichs47,Necas68a,Necas68b,Besov67,
Nitsche81,Oleinik89,Ciarlet10,Ciarlet10b,Horgan75,Horgan83,Duvaut76}.  
A version of Korn's inequality for sequences of gradient young measures
has been obtained in \cite{Bhattacharya91}. 

Korn's inequalities are also crucial in the finite element treatment 
of problems in solid mechanics with non-conforming or discontinuous Galerkin methods.
Piecewise Korn's inequalities subordinate to the mesh 
and involving jumps across element boundaries 
are investigated, e.g., in \cite{Brenner04,Neff_BV}. 
An interesting special case of Korn's first inequality 
with non-standard boundary conditions
and for non-axi-symmetric domains with applications 
in statistical mechanics has been treated in \cite{Villani02}.

Ciarlet \cite{Ciarlet98a,Ciarlet98b,Ciarlet05} has shown 
how to extend Korn's inequalities to curvilinear coordinates 
in Euclidean space which has applications in shell theory. 
It is possible to extend such generalizations to more general 
Riemannian manifolds \cite{Jost02}. 
Korn's inequalities for thin domains with uniform constants 
have been investigated, e.g., in \cite{Lewicka10}.

Korn's inequalities appear in the treatment of the Navier-Stokes model as well,
since with the fluid velocity $v$ in the Eulerian description 
the rate of the deformation tensor is given by $\sym\na v$
which controls the viscous forces generated due to shearing motion.
In this case, Korn's inequality acts in a geometrically exact description 
of the fluid motion and not just for the approximated linearized treatment 
as in linearized elasticity.

\subsection{Further generalizations of Korn's inequalities}

\subsubsection{Poincar\'e-Korn type estimates}

As already mentioned, 
it is well known that there are no $\Woom{1}$-versions 
of Korn's inequalities \cite{Ornstein62,Conti05a}. 
However, it is still possible to obtain a bound of the 
$\Lpom$-norm of a vector field $v$ even for $p=1$ in terms of controlling 
the strain $\sym\na v$ in some sense\footnote{And this 
is indeed the type of inequality a la Poincar\'e's estimate 
that Korn intended to prove \cite[p.707]{Korn09}.}. 
More precisely, let as usual $\BDom$ denote the space of bounded deformations,
i.e., the space of all vector fields $v\in\Loom$ 
such that all components of the tensor (matrix) $\sym\na v$ 
(defined in the distributional sense) 
are measures with finite total variation. 
Then, the total variation measure of the distribution $\sym\na v$ 
for a vector field $v\in\Loom$ is defined by
$$|\sym\na v|(\om):=\sup_{\substack{\Phi\in\Cocom\\\norm{\Phi}_{\Liom}\leq1}}
|\scpom{v}{\Div\sym\Phi}|$$
and $|\sym\na v|(\om)=\norm{\sym\na v}_{\Loom}$ holds if $\sym\na v\in\Loom$.
In \cite{Kohn79,Kohn82} the inequalities
\begin{align*}
\exists\,\ck&>0&\forall\,v&\in\BDom\dopo&
\inf_{r\in\RM}\norm{v-r}_{\Loom}&\leq\ck|\sym\na v|(\om),\\
\exists\,\ck&>0&\forall\,v&\in\Lpom,\sym\na v\in\Lqom\dopo&
\inf_{r\in\RM}\norm{v-r}_{\Lpom}&\leq\ck\norm{\sym\na v}_{\Lqom}
\end{align*}
with
$$q\in[1,\infty)\setminus\{3\},\quad
p=\begin{cases}\frac{3q}{3-q}&\text{, }1\leq q<3\\\infty&\text{, }q>3\end{cases}$$
have been proved.
In case the displacement $v$ has vanishing trace on $\ga$ 
one has a Poincar\'e-Korn type inequality for $v\in\BDom$ \cite{Temam83} 
$$\norm{v}_{\Lgen{3/2}{}(\om)}\leq\ck|\sym\na v|(\om).$$
The weaker inequality with $\Loom$-term on the right hand side 
is already proved in \cite[Th.1]{Strauss71}.
Moreover, as shown in \cite[Th.II.2.4]{Temam83} 
it is clear that $\BDom$ is compactly embedded 
into $\Lpom$ for any $1\leq p<3/2$.

Considering Korn's second inequality one obtains,
again via Rellich's selection theorem,
the compact embedding of
$$\Som:=\set{v\in\Ltom}{\sym\na v\in\Ltom}$$
into $\Ltom$ provided that the `regularity result' $\Som\subset\Hoom$ holds, 
as already mentioned. 
In less regular domains, e.g., domains with cusps, 
Korn's second inequality and the embedding $\Som\subset\Hoom$ may fail, 
for counterexamples see \cite{Weck94,Geymonat98}. 
Weck \cite{Weck94} has shown that, however, 
compact embedding into $\Ltom$, i.e.,
the {\bf elastic compactness property} (ECP)\footnote{Here, 
we have the same situation as in the Maxwell case, 
see the MCP and the MI.}, 
this is, the embedding 
\begin{align}
\mylabel{compembWeck}
\Som\hookrightarrow\Ltom
\end{align}
is compact, still holds true, without 
the intermediate $\Hoom$-estimate.
Therefore, once more by a usual indirect argument, 
also in irregular (bounded) domains one has always the estimate
$$\normLtom{v}\leq\ck\normLtom{\sym\na v}$$
for all $v\in\Som\cap\Szom^{\bot}$, 
where $\Szom:=\set{v\in\Som}{\sym\na v=0}$.
Note that $\Szom$ is finite dimensional\footnote{Compare 
with $\harmdi$ in \eqref{dirichletvecdef}.} 
due to the compact embedding \eqref{compembWeck}.
Moreover, we have $\RM\subset\Szom$ but equality is not clear.

Extensions of Korn's inequalities to non-smooth domains 
and weighted versions for unbounded domains can be found in 
\cite{Oleinik88,Nazarov97,Nazarov06,Duran04,Duran06,Duran06b}. 
Korn's inequalities in Orlicz spaces are treated, e.g., 
in \cite{Fuchs10,Diening11}. A reference for Korn's inequality 
for perforated domains and homogenization theory is \cite{Nazarov09}.

\subsubsection{Generalization to weaker strain measures}

Also the second Korn's inequality can be generalized 
by using the trace free infinitesimal deviatoric strain measure. 
It holds 
$$\norm{v}_{\Hoom}\leq\ck\big(\normLtom{v}+\normLtom{\dev\sym\na v}\big)$$
for all $v\in\Hoom$. For proofs see 
\cite{Dain06,Neff_JeongMMS08,Reshetnyak70,Reshetnyak94} 
and \cite{Fuchs10,Fuchs10b,Fuchs10c}.
This version has found applications for Cosserat models 
and perfect plasticity \cite{Fuchs09}. 

Another generalization concerns the situation, 
where a dislocation based motivated generalized strain
$\sym(\na vF_p^{-1})$ is controlled. 
Such cases arise naturally 
when considering finite elasto-plasticity 
based on the multiplicative decomposition 
$F=F_eF_p$ of the deformation gradient 
into elastic and plastic parts \cite{Neff01c,Neff01d} 
or in elasticity problems with structural changes
\cite{Klawonn_Neff_Rheinbach_Vanis09,Neff_micromorphic_rse_05} 
and shell models \cite{Neff_plate05_poly}. 
In case of plasticity, $F_p:\om\to\rttt$ 
is the plastic deformation related to pure dislocation motion. 
The first result under the assumptions that $\det F_p\geq\mu>0$
and $F_{p}$ is sufficiently smooth, i.e., 
$F_{p},F_{p}^{-1},\Curl F_p\in\Co(\omb)$,
has been given by Neff in \cite{Neff00b}. 
In fact
\begin{align}
\mylabel{korn_neff}
\norm{v}_{\Hoom}\leq\ck\normLtom{\sym(\na vF_{p}^{-1})}
\end{align}
holds for all $v\in\Hocgatom$ with $\ck$ depending on $F_{p}$.
This inequality has been generalized 
to mere continuity and invertibility of $F_p$ in \cite{pompekorn},
while it is also known that some sort of smoothness 
of $F_p$ beyond $\Liom$-control is necessary, 
see \cite{pompekorn,Pompe10,Neff_Pompe13}.

\subsubsection{Korn's inequality and rigidity estimates}

Recently, there has been a revived interest in so called rigidity results, 
which have a close connection to Korn's inequalities. 
With the point-wise representation
\begin{align*}
\dist^2(\na v(x),\so(3))
&=\inf_{\Tskew\in\so(3)}|\na v(x)-\Tskew|^2\\
&=\inf_{\Tskew\in\so(3)}\left(|\sym\na v(x)|^2
+|\skew\na v(x)-\Tskew|^2\right)
=|\sym\na v(x)|^2,
\end{align*}
the infinitesimal rigidity result can be expressed as follows
\begin{align}
\mylabel{infinitesimal_rigidity}
\dist(\na v,\so(3))=0\qimpl v\in\RM.
\end{align}
Korn's first inequality can be seen as a qualitative extension 
of the infinitesimal rigidity result, this is,
for $1<p<\infty$ there exist constants $\ck>0$ such that 
$$\min_{\Tskew\in\so(3)}\norm{\na v-\Tskew}_{\Lpom}
\leq\ck\norm{\sym\na v}_{\Lpom}
=\ck\big(\int_{\om}\dist^p(\na v,\so(3))\dl\big)^{1/p}$$
holds for all $v\in\Woom{p}$, see, e.g., \cite{Valent88}.
As already seen in \eqref{kornwithoutbctwo},
in the Hilbert space case $p=2$ 
the latter inequality can be made explicit with $\Tskew=\Tskew_{\na v}$
and in this form with $\Tskew_{\na v}=0$ it is given
by Friedrichs \cite[p.446]{Friedrichs47} 
and denoted as Korn's inequality in the second case.

The nonlinear version of \eqref{infinitesimal_rigidity} 
is the classical Liouville rigidity result, 
see \cite{Ciarlet88,Reshetnyak67,Reshetnyak94}.
It states that if an elastic body is deformed in such a way 
that its deformation gradient is point-wise a rotation, 
then the body is indeed subject to a rigid motion. 
In mathematical terms we have for smooth maps $\varphi$,
that if $\na\varphi\in\SO(3)$\footnote{$\SO(3)$ 
denotes the space of orthogonal matrices with determinant $1$.} 
almost everywhere then $\na\varphi$ is constant, i.e.,
\begin{align}
\mylabel{rigidity_ciarlet}
\dist(\na\varphi,\SO(3))=0\qimpl\varphi(x)=R x+b,\quad R\in\SO(3),\,b\in\rt.
\end{align}
The optimal quantitative version of Liouville's rigidity result 
has been derived by Friesecke, James and M\"uller in \cite{Mueller02}. 
We have
\begin{align}
\mylabel{rigidity_friesecke}
\min_{R\in\SO(3)}\big(\int_{\om}\dist^p(\na\varphi,R)\dl\big)^{1/p}
&\leq\ck\big(\int_{\om}\dist^p(\na\varphi,\SO(3))\dl\big)^{1/p}.
\intertext{As a consequence, if the deformation gradient is close to rotations,
then it is in fact close to a unique rotation. 
A generalization to fracturing materials is stated in \cite{Chambolle07}. 
It is possible to infer a nonlinear Korn's inequality 
from \eqref{rigidity_friesecke}, i.e.,}
\nonumber
\normLtom{\na\varphi-\id}
&\leq\ck\normLtom{\trans{(\na\varphi)}\na\varphi-\id}
\intertext{for all $\varphi\in\Woom{4}$ with $\varphi=\id$ on $\ga$ and $\det\na\varphi>0$,
see \cite{Mardare11} for more general statements. 
Another quantitative generalization of Liouville's rigidity result
is the following: 
For all differentiable orthogonal tensor fields $R:\om\to\SO(3)$}
\mylabel{curlest}
|\na R|&\leq c\,|\Curl R|
\intertext{holds point-wise \cite{Neff_curl06}.
From \eqref{curlest} we may also recover 
\eqref{rigidity_ciarlet} by assuming $R=\na\varphi$.
It extends the simple inequality for differentiable skew-symmetric 
tensor fields $A:\om\to\so(3)$}
\mylabel{trivcurlest}
|\na A|&\leq c\,|\Curl A|
\end{align}
to $\SO(3)$, i.e., 
to finite rotations \cite{Neff_curl06}.

After this introductory remarks
we turn to the main part of our contribution.

\section{Definitions and preliminaries}\mylabel{defsec}

Let $\om$ be a bounded domain in $\rt$
with Lipschitz boundary $\ga:=\dom$.
Moreover, let $\gat$ be a relatively open subset of $\ga$ 
separated from $\gan:=\dom\setminus\ol{\gat}$
by a Lipschitz curve. 
For details and exact definitions see \cite{jochmanncompembmaxmixbc}.

\subsection{Functions and vector fields}

The usual Lebesgue spaces of square integrable functions, 
vector or tensor fields on $\om$ 
with values in $\rz$, $\rt$ or $\rttt$, respectively,
will be denoted by $\Ltom$.
Moreover, we introduce the standard Sobolev spaces
\begin{align*}
\Hgom&=\set{u\in\Ltom}{\grad u\in\Ltom},&\grad&=\na,\\
\Hcom&=\set{v\in\Ltom}{\curl v\in\Ltom},&\curl&=\na\times,\\
\Hdom&=\set{v\in\Ltom}{\div v\in\Ltom},&\div&=\na\cdot
\end{align*}
of functions $u$ or vector fields $v$, respectively.
$\Hgom$ is usually denoted by $\Hgen{1}{}{}(\om)$.
Furthermore, we introduce their closed subspaces 
$$\Hgcgatom=\Hocgatom,\quad\Hccgatom,\quad\Hdcganom$$ 
as completion under the respective graph norms of
the scalar valued space $\Cicgatom$ 
and the vector valued spaces $\Cicgatom$, $\Cicganom$, where
$$\Cic(\gamma;\om):=\set{u\in\Ci(\ol{\om})}{\dist(\supp u,\gamma)>0},\quad
\gamma\in\{\ga,\gat,\gan\}.$$
In the latter Sobolev spaces, by Gau{\ss}' theorem
the usual homogeneous scalar, tangential and normal boundary conditions
$$u|_{\gat}=0,\quad\nu\times v|_{\gat}=0,\quad\nu\cdot v|_{\gan}=0$$
are generalized, where $\nu$ denotes 
the outward unit normal at $\ga$.\footnote{Note that 
$\nu\times v|_{\gat}=0$ is equivalent to $\tau\cdot v|_{\gat}=0$
for all tangential vector fields $\tau$ at $\gat$.}
If $\gat=\ga$ (and $\gan=\emptyset$) resp.
$\gat=\emptyset$ (and $\gan=\ga$)
we obtain the usual Sobolev-type spaces and write
$$\Hgcom=\Hocom,\quad\Hccom,\quad\Hdom$$ 
resp.
$$\Hgom=\Hoom,\quad\Hcom,\quad\Hdcom.$$ 
Furthermore, we need the spaces of irrotational or solenoidal vector fields
\begin{align*}
\Hczom&:=\set{v\in\Hcom}{\curl v=0},\\
\Hdzom&:=\set{v\in\Hdom}{\div v=0},\\
\Hcczgatom&:=\set{v\in\Hccgatom}{\curl v=0},\\
\Hdczganom&:=\set{v\in\Hdcganom}{\div v=0},
\end{align*}
where the index $0$ indicates vanishing $\curl$ or $\div$, respectively.
All these spaces are Hilbert spaces. 
In classical terms, e.g., a vector field $v$ belongs to 
$\Hcczgatom$ resp. $\Hdczganom$, if
$$\curl v=0,\quad\nu\times v|_{\gat}=0\qquad\text{resp.}\qquad
\div v=0,\quad\nu\cdot v|_{\gan}=0.$$

In \cite{jochmanncompembmaxmixbc} the crucial compact embedding
\begin{align}
\mylabel{maxc}
\Hccgatom\cap\Hdcganom\hookrightarrow\Ltom
\end{align}
has been proved, which we refer to as Maxwell compactness property (MCP).
The generalization to $\rN$ or even to Riemannian manifolds
using the calculus of differential forms can be found in \cite{kuhndiss}
or \cite{jakabmitreairinamariusfinensolhodgedeco}.

A first immediate consequence of \eqref{maxc} is that the space
of so called `harmonic Dirichlet-Neumann fields'
\begin{align}
\mylabel{dirichletvecdef}
\harmdi:=\Hcczgatom\cap\Hdczganom
\end{align}
is finite dimensional, 
since by \eqref{maxc} the unit ball is compact in $\harmdi$. 
In fact, if $\gat=\emptyset$ resp. $\gat=\ga$, 
its dimension equals the first resp. second 
Betti number of $\om$, see \cite{picardboundaryelectro}.
In classical terms we have $v\in\harmdi$ if
$$\curl v=0,\quad\div v=0,\quad
\nu\times v|_{\gat}=0,\quad\nu\cdot v|_{\gan}=0.$$

By an usual indirect argument
we achieve another immediate and important consequence:

\begin{lem}
\mylabel{poincaremax}
{\sf(Maxwell Estimate for Vector Fields)}
There exists a positive constant $\cm$, 
such that for all $v\in\Hccgatom\cap\Hdcganom\cap\harmdi^{\bot}$
$$\normLtom{v}\leq 
\cm\big(\normLtom{\curl v}^2+\normLtom{\div v}^2\big)^{1/2}.$$
\end{lem}

There are two options to get estimate on $\Hccgatom\cap\Hdcganom$.

\begin{cor}
\mylabel{poincaremaxcor}
{\sf(Maxwell Estimate for Vector Fields)}
There exists a positive constant $\cm$, 
such that for all $v\in\Hccgatom\cap\Hdcganom$
\begin{align*}
\normLtom{(\id-\pi)v}
&\leq\cm\big(\normLtom{\curl v}^2+\normLtom{\div v}^2\big)^{1/2},\\
\normLtom{v}
&\leq\cm\big(\normLtom{\curl v}^2+\normLtom{\div v}^2+\normLtom{\pi v}^2\big)^{1/2}.
\end{align*}
Here $\pi:\Ltom\to\harmdi$ denotes the $\Ltom$-orthogonal projection 
onto Dirichlet-Neumann fields and can be expressed explicitly by
$$\pi v:=\sum_{\ell=1}^{L}\scpLtom{v}{d^{\ell}}d^{\ell},\quad
\normLtom{\pi v}^2=\sum_{\ell=1}^{L}|\scpLtom{v}{d^{\ell}}|^2,$$
where $L:=\dim\harmdi$ and $(d^{\ell})_{\ell=1}^{L}$ an $\Ltom$-orthonormal basis of $\harmdi$.
\end{cor}

Here, we denote by $\bot$ the orthogonal complement in $\Ltom$.
As shown in \cite{jochmanncompembmaxmixbc} as well we have
\begin{align*}
\grad\Hgcgatom^{\bot}&=\Hdczganom,&
\curl\Hccganom^{\bot}&=\Hcczgatom,
\intertext{which implies}
\ol{\grad\Hgcgatom}&=\Hdczganom^{\bot},&
\ol{\curl\Hccganom}&=\Hcczgatom^{\bot},
\intertext{where the closures are taken in $\Ltom$. Since} 
\grad\Hgcgatom&\subset\Hcczgatom,&
\curl\Hccganom&\subset\Hdczganom
\end{align*}
we obtain by the projection theorem the Helmholtz decompositions
\begin{align*}
\Ltom
&=\ol{\grad\Hgcgatom}\oplus\Hdczganom\\
&=\Hcczgatom\oplus\ol{\curl\Hccganom}\\
&=\ol{\grad\Hgcgatom}\oplus\harmdi\oplus\ol{\curl\Hccganom},
\end{align*}
where $\oplus$ denotes the $\Ltom$-orthogonal sum.
Using an indirect argument, the space $\grad\Hgcgatom$ is already closed 
by variants of Poincar\'e's estimate, i.e.,
\begin{align}
\mylabel{poincareestcl}
\gat&\neq\emptyset:&
\exists\,\cp&>0&
\forall\,u&\in\Hgcgatom&
\normLtom{u}&\leq\cp\normLtom{\grad u},\\
\gat&=\emptyset:&
\exists\,\cp&>0&
\forall\,u&\in\Hgom\cap\{1\}^{\bot}&
\normLtom{u}&\leq\cp\normLtom{\grad u},\nonumber
\end{align}
which are implied by the compact embeddings (Rellich's selection theorems)
\begin{align}
\mylabel{rellichc}
\Hgcgatom\hookrightarrow\Ltom,\quad
\Hgom\hookrightarrow\Ltom.
\end{align}
Analogously to Corollary \ref{poincaremaxcor} we also have 
for $\gat=\emptyset$ and all $u\in\Hgom$
\begin{align*}
\normLtom{u-\alpha_{u}}
&\leq\cp\normLtom{\grad u},\quad
\alpha_{u}:=\lambda(\om)^{-1}\scpLtom{u}{1}=\oint_{\om}u\dl\in\rz,\\
\normLtom{u}&
\leq\cp(\normLtom{\grad u}^2+\normLtom{\alpha_{u}}^2)^{1/2}.
\end{align*}
Interchanging $\gat$ and $\gan$ in the second equation
of the latter Helmholtz decompositions 
and applying this Helmholtz decompositions to $\Hccganom$
yields the refinement
$$\curl\Hccganom=\curl\big(\Hccganom\cap\ol{\curl\Hccgatom}\big).$$
Now, by Lemma \ref{poincaremax} we see that
$\curl\Hccganom$ is closed as well.
We have:

\begin{lem}
\mylabel{helmdeco}
{\sf(Helmholtz Decompositions for Vector Fields)}
The orthogonal decompositions
\begin{align*}
\Ltom
&=\grad\Hgcgatom\oplus\Hdczganom\\
&=\Hcczgatom\oplus\curl\Hccganom\\
&=\grad\Hgcgatom\oplus\harmdi\oplus\curl\Hccganom
\intertext{hold. Moreover,}
\curl\Hccganom
&=\curl\big(\Hccganom\cap\curl\Hccgatom\big).
\end{align*} 
\end{lem}
 
\subsection{Tensor fields}

We extend our calculus to $(3\times3)$-tensor (matrix) fields.
For vector fields $v$ with components in $\Hgom$
and tensor fields $\T$ with rows in $\Hcom$ resp. $\Hdom$, i.e.,
$$v=\dvec{v_{1}}{v_{2}}{v_{3}},\quad v_{n}\in\Hgom,\quad
\T=\dvec{\trans{\T_{1}}}{\trans{\T_{2}}}{\trans{\T_{3}}},\quad
\T_{n}\in\Hcom\text{ resp. }\Hdom$$
we define
$$\Grad v:=\dvec{\trans{\grad}v_{1}}{\trans{\grad}v_{2}}{\trans{\grad}v_{3}}=J_{v},\quad
\Curl\T:=\dvec{\trans{\curl}\T_{1}}{\trans{\curl}\T_{2}}{\trans{\curl}\T_{3}},\quad
\Div\T:=\dvec{\div\T_{1}}{\div\T_{2}}{\div\T_{3}},$$ 
where $J_{v}$ denotes the Jacobian of $v$ 
and $\trans{}$ the transpose.
We note that $v$ and $\Div\T$ are vector fields, whereas
$\T$, $\Curl\T$ and $\Grad v$ are tensor fields.
The corresponding Sobolev spaces will be denoted by
$$\HGom,\quad
\HCom,\quad\HCzom,\quad
\HDom,\quad\HDzom$$
and
$$\HGcgatom,\quad
\HCcgatom,\quad\HCczgatom,\quad
\HDcganom,\quad\HDczganom.$$

Now, we present our three crucial tools to prove our main estimate.
First we have obvious consequences 
from Lemmas \ref{poincaremax} and \ref{helmdeco}:

\begin{cor}
\mylabel{poincaremaxten}
{\sf(Maxwell Estimate for Tensor Fields)}
The estimate
$$\normLtom{\T}\leq 
\cm\big(\normLtom{\Curl\T}^2+\normLtom{\Div\T}^2\big)^{1/2}$$
holds for all tensor fields $\T\in\HCcgatom\cap\HDcganom\cap(\harmdi^3)^{\bot}$.
Furthermore, the analogue of Corollary \ref{poincaremaxcor} holds as well.
\end{cor}

Here, $\T\in\harmdi^3$ if
$\trans{\T}=[\T_{1}\,\T_{2}\,\T_{3}]$ with $\T_{m}\in\harmdi$ for $m=1,\dots,3$.

\begin{cor}
\mylabel{helmdecoten}
{\sf(Helmholtz Decomposition for Tensor Fields)}
The orthogonal decompositions
\begin{align*}
\Ltom
&=\Grad\HGcgatom\oplus\HDczganom\\
&=\HCczgatom\oplus\Curl\HCcganom\\
&=\Grad\HGcgatom\oplus\harmdi^3\oplus\Curl\HCcganom
\intertext{hold. Moreover,}
\Curl\HCcganom
&=\Curl\big(\HCcganom\cap\Curl\HCcgatom\big).
\end{align*} 
\end{cor}

The third important tool is Korn's first inequality
and a variant which meets our needs is the next lemma.

\begin{lem}
\mylabel{korn}
{\sf(Korn's First Inequality: Standard Version)}
There exists a constant $\cks>0$, such that the following holds:
\begin{itemize}
\item[\bf(i)] If $\gat\neq\emptyset$ then
\begin{align}
\mylabel{firstkornstandard}
(1+\cp^2)^{-1/2}\norm{v}_{\Hoom}\leq\normLtom{\Grad v}\leq\cks\normLtom{\sym\Grad v}
\end{align}
holds for all vector fields $v\in\HGcgatom$.
\item[\bf(ii)] If $\gat=\emptyset$, 
then the inequalities \eqref{firstkornstandard}
hold for all vector fields $v\in\HGom$ with $\Grad v\bot\so(3)$ and $v\bot\,\rt$.
Moreover, the second inequality of \eqref{firstkornstandard} holds
for all vector fields $v\in\HGom$ with $\Grad v\bot\so(3)$.
For all $v\in\HGom$
\begin{align}
\mylabel{firstkornstandardempty}
(1+\cp^2)^{-1/2}\norm{v-r_{v}}_{\Hoom}\leq\normLtom{\Grad v-\Tskew_{\Grad v}}
\leq\cks\normLtom{\sym\Grad v}
\end{align}
holds, where $r_{v}\in\RM$ and $\Tskew_{\Grad v}=\Grad r_{v}$ 
are given by $r_{v}(x):=\Tskew_{\Grad v}x+b_{v}$ and
$$\Tskew_{\Grad v}:=\skew\oint_{\om}\Grad v\dl\in\so(3),\quad
b_{v}:=\oint_{\om}v\dl-\Tskew_{\Grad v}\oint_{\om}x\dl_{x}\in\rt.$$
We note $v-r_{v}\bot\,\rt$ and $\Grad(v-r_{v})=\Grad v-\Tskew_{\Grad v}\bot\so(3)$.
\end{itemize}
\end{lem}

\begin{proof}
As already mentioned in the introduction,
the assertions are easy consequences of Korn's second inequality
and Rellich's selection theorem for $\Hoom$.
\end{proof}

\begin{rem}
\mylabel{kornrem}
Note that $\Tskew_{\Grad v}=\pi_{\so(3)}\Grad v$, where
$\pi_{\so(3)}:\Ltom\to\so(3)$ denotes the $\Ltom$-orthogonal projection 
onto $\so(3)$. Thus, the assertion
$$\Grad(v-r_{v})=\Grad v-\Tskew_{\Grad v}=(\id-\pi_{\so(3)})\Grad v\bot\so(3)$$
is trivial. Moreover, generally for $\T\in\Ltom$
\begin{align}
\mylabel{defTpi}
\pi_{\so(3)}\T:=\Tskew_{\T}:=\skew\oint_{\om}\T\dl\in\so(3)
\end{align}
holds. Equivalent to \eqref{firstkornstandardempty} we have for all $v\in\HGom$
\begin{align*}
(1+\cp^2)^{-1/2}\norm{v}_{\Hoom}
&\leq\big(\normLtom{\na v}^2
+\normLtom{a_{v}}^2\big)^{1/2}\\
&\leq\ck\big(\normLtom{\sym\na v}^2
+\normLtom{\Tskew_{\Grad v}}^2+\normLtom{a_{v}}^2\big)^{1/2}\\
&\leq\ck\big(\normLtom{\sym\na v}^2
+\norm{r_{v}}_{\Hoom}^2\big)^{1/2}
\end{align*}
with
$$a_{v}=\pi_{\rt}v:=\oint_{\om}v\dl\in\rt,$$
where $\pi_{\rt}:\Ltom\to\rt$ denotes the $\Ltom$-orthogonal projection 
onto $\rt$. For details, we refer to the appendix.
\end{rem}

\section{Main results}

We start with generalizing Korn's first inequality 
from gradient tensor fields to merely irrotational tensor fields.

\subsection{Extending Korn's first inequality to irrotational tensor fields}

\begin{lem}
\mylabel{constlem}
Let $\gat\neq\emptyset$ and $u\in\Hgom$ with $\grad u\in\Hcczgatom$.
Then, $u$ is constant on any connected component of $\gat$.
\end{lem}

\begin{proof}
It is sufficient to show that $u$ is locally constant.
Let $x\in\gat$ and $B_{2r}:=B_{2r}(x)$ be an open ball of radius $2r>0$
around $x$ such that $B_{2r}$ is covered by a Lipschitz-chart domain
and $\ga\cap B_{2r}\subset\gat$.
Moreover, we pick some $\varphi\in\Cic(B_{2r})$ with
$\varphi|_{B_{r}}=1$.
Then $\varphi\grad u\in\Hcgen{}{\circ}{\om\cap B_{2r}}$.
Thus, the extension by zero $v$ of $\varphi\grad u$ to $B_{2r}$
belongs to $\Hcgen{}{}{B_{2r}}$. 
Hence, $v|_{B_{r}}\in\Hcgen{0}{}{B_{r}}$.
Since $B_{r}$ is simply connected, 
there exists a $\tilde{u}\in\Hggen{}{}{B_{r}}$
with $\grad\tilde{u}=v$ in $B_{r}$. 
In $B_{r}\setminus\ol{\om}$ we have $v=0$.
Therefore, $\tilde{u}|_{B_{r}\setminus\ol{\om}}=\tilde{c}$ 
with some $\tilde{c}\in\rz$.
Moreover, $\grad u=v=\grad\tilde{u}$ holds in $B_{r}\cap\om$,
which yields $u=\tilde{u}+c$ in $B_{r}\cap\om$ with some $c\in\rz$.
Finally, $u|_{B_{r}\cap\gat}=\tilde{c}+c$ is constant.
\end{proof}

\begin{lem}
\mylabel{kornlem}
{\sf(Korn's First Inequality: Tangential Version)}
Let $\gat\neq\emptyset$.
There exists a constant $\ckt\geq\cks$, such that
$$\normLtom{\Grad v}\leq\ckt\normLtom{\sym\Grad v}$$
holds for all vector fields $v\in\HGom$ with $\Grad v\in\HCczgatom$.
\end{lem}

In classical terms, $\Grad v\in\HCczgatom$ 
means that the restricted tangential traces
$\nu\times\grad v_{n}|_{\gat}$ vanish, i.e.,
$\grad v_{n}=\na v_{n}$, $n=1,\dots,3$, are normal at $\gat$. 
In other words, $\tau\cdot\na v_n|_{\gat}=0$ 
for all tangential vectors fields $\tau$ on $\gat$.\\

\begin{proof}
Let $\tilde{\ga}\neq\emptyset$ 
be a relatively open connected component of $\gat$.
Applying Lemma \ref{constlem} to each component of $v$,
there exists a constant vector $c_{v}\in\rt$ 
such that $v-c_{v}$ belongs to $\HGgen{}{\circ}{\tilde{\ga},\om}$.
Then, Lemma \ref{korn} (i)
(with $\gat=\tilde{\ga}$ and a possibly larger $\ckt$) 
completes the proof.
\end{proof}

\begin{defi}
\mylabel{domdef}
$\om$ is called `sliceable',
if there exist $J\in\nz$ and $\om_{j}\subset\om$, $j=1,\dots,J$, 
such that $\om\setminus(\om_{1}\cup\ldots\cup\om_{J})$ is a nullset and for $j=1,\dots,J$
\begin{itemize}
\item[\bf(i)] 
$\om_{j}$ are open, disjoint and simply connected Lipschitz subdomains of $\om$,
\item[\bf(ii)] 
$\gatj:=\interior_{\mathsf{rel}}(\ol{\om_{j}}\cap\gat)\neq\emptyset$, if $\gat\neq\emptyset$.
\end{itemize}
\end{defi}

Here, $\interior_{\mathsf{rel}}$ denotes the interior with respect to
the topology on $\ga$.

\begin{rem}
\mylabel{slicerem}
Assumptions of this type are not new, see e.g.
\cite[p.836]{amrouchebernardidaugegiraultvectorpot} or \cite[p.3]{Ciarlet10b}.
From a practical point of view, 
\ul{all} domains considered in applications are sliceable,
but it is not clear whether every Lipschitz domain is already sliceable. 
\end{rem}

\begin{figure}
\center
\includegraphics[scale=0.08]{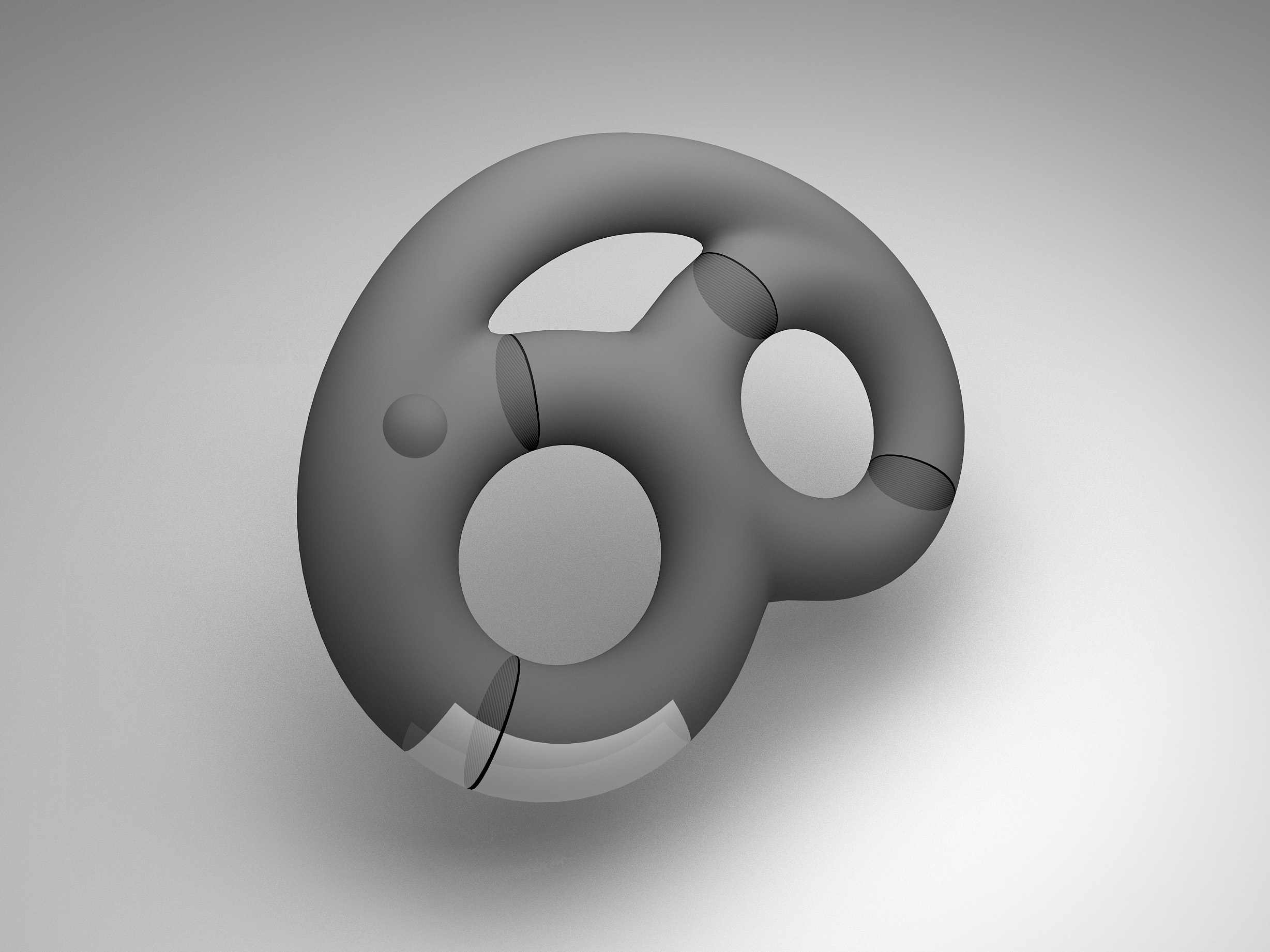}
\includegraphics[scale=0.08]{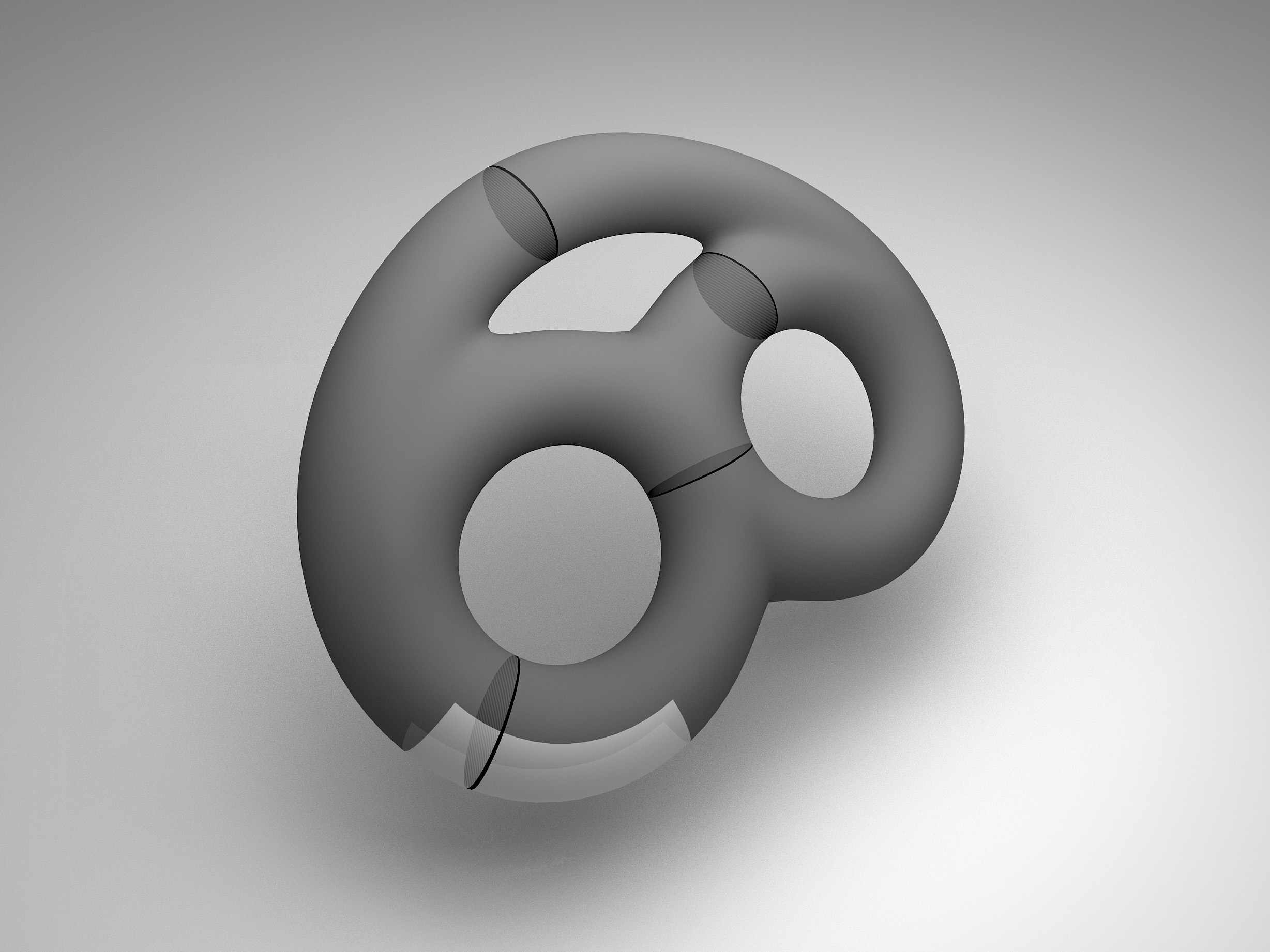}
\caption{Two ways to cut a sliceable domain into two ($J=2$) subdomains.
Roughly speaking, a domain is sliceable
if it can be `cut' into finitely many simply connected Lipschitz `pieces' $\om_{j}$, i.e., 
any closed curve inside some piece $\om_{j}$
is homotop to a point, this is, one has to cut all handles. 
Holes inside $\om$ are permitted since we are in 3D.}
\mylabel{brezel}
\end{figure}

\begin{lem}
\mylabel{genkornlem}
{\sf(Korn's First Inequality: Irrotational Version)}
Let $\om$ be sliceable.
There exists $\ck\geq\ckt>0$, 
such that the following inequalities hold:
\begin{itemize}
\item[\bf(i)] 
If $\gat\neq\emptyset$, then for all tensor fields $\T\in\HCczgatom$
\begin{align}
\mylabel{estTsymT}
\normLtom{\T}\leq\ck\normLtom{\sym\T}.
\end{align}
\item[\bf(ii)] 
If $\gat=\emptyset$, then for all tensor fields $\T\in\HCzom$
there exists a piece-wise constant skew-symmetric tensor field $\Tskew$ such that
\begin{align*}
\normLtom{\T-\Tskew}
&\leq\ck\normLtom{\sym\T},\\
\normLtom{\T}
&\leq\ck\big(\normLtom{\sym\T}^2+\normLtom{\Tskew}^2\big)^{1/2}.
\end{align*}
\item[\bf(ii')] 
If $\gat=\emptyset$ and $\om$ is additionally simply connected, 
then (ii) holds with the uniquely determined 
constant skew-symmetric tensor field $\Tskew:=\Tskew_{\T}=\pi_{\so(3)}\T$
given by \eqref{defTpi}.
Moreover, $\T-\Tskew_{\T}\in\HCzom\cap\so(3)^{\bot}$ and
$\Tskew_{\T}=0$ if and only if $\T\bot\so(3)$. 
Thus, \eqref{estTsymT} holds for all $\T\in\HCzom\cap\so(3)^{\bot}$.
\end{itemize}
\end{lem}

Again we note that in classical terms 
a tensor $\T\in\HCczgatom$ is irrotational 
and the vector field $\T\tau|_{\gat}$ vanishes 
for all tangential vector fields $\tau$ at $\gat$.

\begin{rem}
\mylabel{remgarroni}
Without proof the last part of the result Lemma \ref{genkornlem} (ii') 
has been used implicitly in \cite{Garroni10}.
The authors of \cite{Garroni10} neglect the problems 
caused by non-simply connected domains.
See also our discussion in \cite{neffpaulywitschgenkornrtzap}.
\end{rem}

\begin{proof}
Let $\gat\neq\emptyset$. According to Definition \ref{domdef} 
we decompose $\om$ into $\om_{1}\cup\ldots\cup\om_{J}$.
Let $\T\in\HCczgatom$ and $1\leq j\leq J$.
Then, the restriction $\T_{j}:=\T|_{\om_{j}}$ 
belongs to $\HCgen{0}{}{\om_{j}}$.
Picking a sequence $(\T^{\ell})\subset\Cic(\gat;\om)$ 
converging to $\T$ in $\HCom$, we see that
$(\T^{\ell}|_{\ol{\om_{j}}})\subset\Cic(\gatj;\om)$ 
converges to $\T_{j}$ in $\HCgen{}{}{\om_{j}}$.
Thus, $\T_{j}\in\HCgen{0}{\circ}{\gatj,\om_{j}}$.
By definition, each $\om_{j}$ is simply connected.
Therefore, there exist potential vector fields $v_{j}\in\HGgen{}{}{\om_{j}}$
with $\Grad v_{j}=\T_{j}$.
Lemma \ref{kornlem} yields for all $j$
\begin{align*}
\norm{\T_{j}}_{\Lt(\om_j)}
&\leq\cktj\norm{\sym\T_{j}}_{\Lt(\om_j)}
\intertext{with $\cktj>0$. Summing up, we obtain}
\normLtom{\T}
&\leq\ck\normLtom{\sym\T},\quad\ck:=\max_{j=1,\dots,J}\cktj,
\end{align*}
which proves (i). 
Now, we assume $\gat=\emptyset$.
Let $\T\in\HCzom$ and again let $\om$ be decomposed into
$\om_{1}\cup\ldots\cup\om_{J}$ by Definition \ref{domdef}.
Again, since every $\om_{j}$, $j=1,\ldots,J$, is simply connected
and $\T_{j}\in\HCgen{0}{}{\om_{j}}$,
there exist vector fields $v_{j}\in\HGgen{}{}{\om_{j}}$ 
with $\Grad v_{j}=:\T_{j}=\T$ in $\om_{j}$.
By Korn's first inequality, Lemma \eqref{korn} (ii), 
there exist positive $\cksj$ and $\Tskew_{\T_{j}}\in\so(3)$ with
$$\norm{\T_{j}-\Tskew_{\T_{j}}}_{\Lt(\om_{j})}
\leq\cksj\norm{\sym\T_{j}}_{\Lt(\om_{j})},\quad
\Tskew_{\T_{j}}=\skew\oint_{\om_{j}}\T_{j}\dl=\skew\oint_{\om_{j}}\T\dl.$$
We define the piece-wise constant skew-symmetric tensor field $\Tskew$ a.e. by
$\Tskew|_{\om_{j}}:=\Tskew_{\T_{j}}$ and set $\ds\ck:=\max_{j=1,\ldots,J}\cksj$.
Summing up gives (ii). 
We have also proved the first assertion of (ii'),
since we do not have to slice if $\om$ is simply connected.
The remaining assertion of (ii') concerning the projections 
are trivial, since $\pi_{\so(3)}:\Ltom\to\so(3)$ 
is a $\Ltom$-orthogonal projector.
We note that this can be seen also by direct calculations:
To show that $\T-\Tskew_{\T}$ belongs to $\HCzom\cap\so(3)^{\bot}$ 
we note $\Tskew_{\T}\in\HCzom$ and compute for all $\Tskew\in\so(3)$
(compare with \eqref{AvAscp})
\begin{align*}
\scpLtom{\Tskew_{\T}}{\Tskew}
&=\scps{\int_{\om}\T\dl}{\Tskew}_{\rttt}
=\int_{\om}\scp{\T}{\Tskew}_{\rttt}\dl
=\scpLtom{\T}{\Tskew}.
\end{align*}
Hence, $\Tskew_{\T}=0$ implies $\T\bot\so(3)$.
On the other hand, setting $\Tskew:=\Tskew_{\T}$ shows that
$\T\bot\so(3)$ also implies $\Tskew_{\T}=0$.
\end{proof}

\subsection{The new inequality}

From now on, we assume generally that \ul{$\om$ is sliceable}.
For tensor fields $\T\in\HCom$ we define the semi-norm
$\dnorm{\,\cdot\,}$ by
\begin{align}
\mylabel{hsymnormdef}
\dnorm{\T}^2:=\normLtom{\sym\T}^2+\normLtom{\Curl\T}^2.
\end{align}

Our main result is presented in the following theorem.

\begin{theo}
\mylabel{mainlem}
Let $\ch:=\max\{\sqrt{2}\ck,\cm\sqrt{1+2\ck^2}\}$ 
and $\ct:=\sqrt{2}\max\{\ck,\cm(1+\ck)\}\geq\ch$. 
\begin{itemize}
\item[\bf(i)] 
If $\gat\neq\emptyset$, then for all tensor fields $\T\in\HCcgatom$ 
\begin{align}
\mylabel{estTsymTCurlT}
\normLtom{\T}\leq\ch\dnorm{\T}.
\end{align}
\item[\bf(ii)] 
If $\gat=\emptyset$, then for all tensor fields $\T\in\HCom$ 
there exists a piece-wise constant skew-symmetric tensor field $\Tskew$, such that
$$\normLtom{\T-\Tskew}\leq\ct\dnorm{\T},\quad
\normLtom{\T}\leq\sqrt2\max\{\ct,1\}\big(\dnorm{\T}^2+\normLtom{\Tskew}^2\big)^{1/2}.$$
Note that, in general $\Tskew\notin\HCom$.
\item[\bf(ii')] 
If $\gat=\emptyset$ and $\om$ is additionally simply connected, 
then for all tensor fields $\T$ in $\HCom$ 
there exists a uniquely determined constant 
skew-symmetric tensor field $\Tskew=\Tskew_{\T}\in\so(3)$, such that
\begin{align*}
\normLtom{\T-\Tskew_{\T}}
&\leq\ch\dnorm{\T},\\
\normLtom{\T}
&\leq\sqrt2\max\{\ch,1\}\big(\dnorm{\T}^2+\normLtom{\Tskew_{\T}}^2\big)^{1/2},\\
\norm{\T-\Tskew_{\T}}_{\HCom}
&\leq(1+\ch^2)^{1/2}\dnorm{\T},\\
\norm{\T}_{\HCom}
&\leq\sqrt2(1+\ch^2)^{1/2}\big(\dnorm{\T}^2+\normLtom{\Tskew_{\T}}^2\big)^{1/2}.
\end{align*}
and $\Tskew_{\T}=\pi_{\so(3)}\T$ is given by \eqref{defTpi}.
Moreover, $\T-\Tskew_{\T}\in\HCom\cap\so(3)^{\bot}$ and
$\Tskew_{\T}=0$ if and only if $\T\bot\so(3)$. 
Thus, \eqref{estTsymTCurlT} holds for all $\T\in\HCom\cap\so(3)^{\bot}$.
Furthermore, $\Tskew_{\T}$ can be represented by
$$\Tskew_{\T}=\Tskew_{\TR}:=\pi_{\so(3)}\TR=\skew\oint_{\om}\TR\dl\in\so(3),$$
where $\TR$ denotes the Helmholtz projection of $\T$ 
onto $\HCzom$ according to Corollary \ref{helmdecoten}.
\end{itemize}
\end{theo}

\begin{proof}
Let $\gat\neq\emptyset$ and $\T\in\HCcgatom$.
According to Corollary \ref{helmdecoten} we orthogonally decompose 
$$\T=\TR+\TS\in\HCczgatom\oplus\Curl\HCcganom.$$
Then, $\Curl\TS=\Curl\T$ and we observe that $\TS$ belongs to
$$\HCcgatom\cap\Curl\HCcganom
=\HCcgatom\cap\HDczganom\cap(\harmdi^3)^{\bot}.$$
Hence, by Corollary \ref{poincaremaxten} we have
\begin{align}
\mylabel{estpsi}
\normLtom{\TS}&\leq\cm\normLtom{\Curl\T}.
\intertext{Then, by orthogonality, Lemma \ref{genkornlem} (i) for $\TR$ 
and \eqref{estpsi} we obtain}
\normLtom{\T}^2=\normLtom{\TR}^2+\normLtom{\TS}^2
&\leq\ck^2\normLtom{\sym\TR}^2+\normLtom{\TS}^2\nonumber\\
&\leq2\ck^2\normLtom{\sym\T}^2+(1+2\ck^2)\normLtom{\TS}^2\nonumber
\end{align}
and thus $\normLtom{\T}^2\leq\ch^2\dnorm{\T}^2$, which proves (i).

Now, let $\gat=\emptyset$ and $\T\in\HCom$.
First, we show (ii').
We follow in close lines the first part of the proof.
For the convenience of the reader, we repeat the previous 
arguments in this special case. 
According to Corollary \ref{helmdecoten} we orthogonally decompose 
$$\T=\TR+\TS\in\HCzom\oplus\Curl\HCcom.$$
Then, $\Curl\TS=\Curl\T$ and
$$\TS\in\HCom\cap\Curl\HCcom
=\HCom\cap\HDczom\cap(\harmdi^3)^{\bot}.$$
Again, by Corollary \ref{poincaremaxten} we have \eqref{estpsi}.
Note that $\Tskew_{\TR}\in\HCzom$ since $\Tskew_{\TR}\in\so(3)$ is constant.
Then, by orthogonality, Lemma \ref{genkornlem} (ii') applied to $\TR$ and \eqref{estpsi}
\begin{align*}
\normLtom{\T-\Tskew_{\TR}}^2
=\normLtom{\TR-\Tskew_{\TR}}^2+\normLtom{\TS}^2
&\leq\ck^2\normLtom{\sym\TR}^2+\normLtom{\TS}^2\\
&\leq2\ck^2\normLtom{\sym\T}^2+(1+2\ck^2)\normLtom{\TS}^2
\end{align*}
and thus $\normLtom{\T-\Tskew_{\TR}}^2\leq\ch^2\dnorm{\T}^2$.
We need to show $\Tskew_{\T}=\Tskew_{\TR}$ or equivalently $\Tskew_{\TS}=0$.
For this, let $\Tskew\in\so(3)$ and $\TS=\Curl X$ with $X\in\HCcom$. Then
$$\scpLtom{\Tskew_{\TS}}{\Tskew}
=\scps{\int_{\om}\TS\dl}{\Tskew}_{\rttt}
=\scpLtom{\Curl X}{\Tskew}=0.$$
By setting $\Tskew:=\Tskew_{\TS}$, we get $\Tskew_{\TS}=0$.
The proof of (ii') is complete, since all other remaining assertions are trivial.
Finally, we show (ii). For this, we follow the proof of (ii') 
up to the point, where $\Tskew_{\TR}$ came into play.
Now, by Lemma \ref{genkornlem} (ii) for $\TR$ we get a piece-wise constant 
skew-symmetric tensor $\Tskew:=\Tskew_{\TR}$.
We note that in general $\Tskew$ does not belong to $\HCom$ anymore.
Hence, we loose the $\Ltom$-orthogonality $\TR-\Tskew\bot\TS$.
But again, by Lemma \ref{genkornlem} (ii) and \eqref{estpsi}
\begin{align*}
\normLtom{\T-\Tskew}
&\leq\normLtom{\TR-\Tskew}+\normLtom{\TS}
\leq\ck\normLtom{\sym\TR}+\normLtom{\TS}\\
&\leq\ck\normLtom{\sym\T}+(1+\ck)\normLtom{\TS}\\
&\leq\ck\normLtom{\sym\T}+(1+\ck)\cm\normLtom{\Curl\T}
\end{align*}
and thus $\normLtom{\T-\Tskew}\leq\ct\dnorm{\T}$, which proves (ii).
\end{proof}

As easy consequence we obtain:

\begin{theo}
\mylabel{maintheo}
Let $\gat\neq\emptyset$ resp. $\gat=\emptyset$ and $\om$ be simply connected.
Then, on $\HCcgatom$ resp. $\HCom\cap\so(3)^{\bot}$ 
the norms $\norm{\,\cdot\,}_{\HCom}$
and $\dnorm{\,\cdot\,}$ are equivalent.
In particular, $\dnorm{\,\cdot\,}$ is a norm on 
$\HCcgatom$ resp. $\HCom\cap\so(3)^{\bot}$
and there exists a positive constant $c$, such that
$$c\norm{\T}_{\HCom}
\leq\dnorm{\T}
=\big(\normLtom{\sym\T}^2+\normLtom{\Curl\T}^2\big)^{1/2}
\leq\norm{\T}_{\HCom}$$
holds for all $\T$ in $\HCcgatom$ resp. $\HCom\cap\so(3)^{\bot}$.
\end{theo}

\subsection{Consequences and relations to Korn and Poincar\'e}

There are two immediate consequences 
of Theorem \ref{mainlem} and the inclusion
$$\Grad\HGcgatom\subset\HCczgatom$$
if the tensor field $\T$ is either {\it irrotational} or {\it skew-symmetric}.

For irrotational tensor fields $\T$, i.e., 
$\Curl\T=0$ or even $\T=\Grad v$, 
we obtain generalized versions of Korn's first inequality.
E.g., in the case $\gat\neq\emptyset$ we get:

\begin{cor}
\mylabel{genkorn}
{\sf(Korn's First Inequality)}
Let $\gat\neq\emptyset$.
\begin{itemize}
\item[\bf(i)] $\normLtom{\T}\leq\hat{c}\normLtom{\sym\T}$ 
holds for all tensor fields $\T\in\HCczgatom$.
\item[\bf(ii)] $\normLtom{\Grad v}\leq\hat{c}\normLtom{\sym\Grad v}$
holds for all vector fields $v\in\HGom$ with $\Grad v\in\HCczgatom$.
\item[\bf(iii)] $\normLtom{\Grad v}\leq\hat{c}\normLtom{\sym\Grad v}$
holds for all vector fields $v\in\HGcgatom$.
\end{itemize}
\end{cor}

These are different generalized versions of Korn's first inequality.
(iii), i.e., the classical Korn's first inequality from Lemma \ref{korn} (i), 
is implied by (ii), i.e., Lemma \ref{kornlem},
which is implied by (i), i.e., Lemma \ref{genkornlem} (i).
We note $\cks\leq\ckt\leq\ck\leq\ch$ and that in classical terms the boundary condition, 
e.g., in (ii), holds, if $\grad v_{n}=\na v_{n}$, $n=1,\dots,3$, are normal at $\gat$,
which then extends (iii) through the weaker boundary condition.

For skew-symmetric tensors fields we get back Poincar\'e's inequality.
More precisely, we may identify a scalar function $u$ 
with a skew-symmetric tensor field $\T$, i.e.,
$$\T:=\T_{u}:=\dmat{0}{0}{u}{0}{0}{0}{-u}{0}{0}\cong u
\quad\text{and hence}\quad 
\Curl\T=\dmat{\p_{2}u}{-\p_{1}u}{0}{0}{0}{0}{0}{-\p_{3}u}{\p_{2}u}.$$
Now, $\Curl\T$ is as good as $\grad u$, see \eqref{trivcurlest} and
$$\nu\times\T|_{\gat}=\dmat{\nu_{2}u|_{\gat}}{-\nu_{1}u|_{\gat}}
{0}{0}{0}{0}{0}{-\nu_{3}u|_{\gat}}{\nu_{2}u|_{\gat}}=0\qequi u|_{\gat}=0.$$
E.g., in the case $\gat\neq\emptyset$ we get by Theorem \ref{mainlem} (i):

\begin{cor}
\mylabel{genpoincare}
{\sf(Poincar\'e's Inequality)}
Let $\gat\neq\emptyset$. 
For all special skew-symmetric tensor fields 
$\T=\T_{u}$ in $\HCcgatom$, i.e., 
for all functions $u\in\Hgcgatom$ with $u\cong\T$, 
$$\normLtom{u}\leq\hat{c}\normLtom{\grad u}.$$
\end{cor}

\begin{proof}
We have $\T\in\HCcgatom$, if and only if $u\in\Hgcgatom$. Moreover,
$$2\normLtom{u}^2=\normLtom{\T}^2\leq\hat{c}^2\normLtom{\Curl\T}^2
\leq2\hat{c}^2\normLtom{\grad u}^2$$
holds.
\end{proof}

We note that the latter Corollary also remains true
for general skew-symmetric tensor fields $\T\in\HCcgatom$
and vector fields $v\in\HGcgatom$ with
$$\T=\dmat{0}{-v_{1}}{v_{2}}{v_{1}}{0}{-v_{3}}{-v_{2}}{v_{3}}{0}\cong v.$$

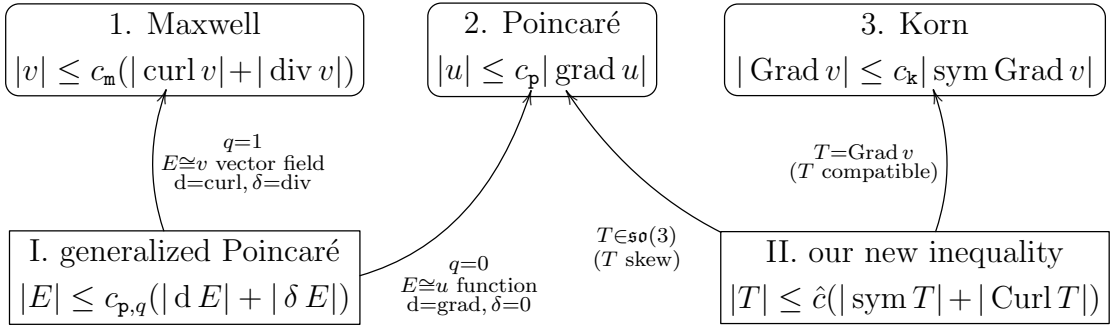
\begin{figure}
$$\xymatrix{
*+[F-:<5pt>]{\begin{minipage}{4.5cm}
\begin{center}
1. Maxwell\\[1mm]
$|v|\leq\cm(|\curl v|+|\div v|)$
\end{center}
\end{minipage}}
&
*+[F-:<5pt>]{\begin{minipage}{2.8cm}
\begin{center}
2. Poincar\'e\\[1mm]
$|u|\leq\cp|\grad u|$
\end{center}
\end{minipage}}
&
*+[F-:<5pt>]{\begin{minipage}{4.8cm}
\begin{center}
3. Korn\\[1mm]
$|\Grad v|\leq\ck|\sym\Grad v|$
\end{center}
\end{minipage}}
\\
&&
\\
*+[F-:<0pt>]{\begin{minipage}{4.4cm}
\begin{center}
I. generalized Poincar\'e\\[1mm]
$|E|\leq\cpq(|\ed E|+|\cd E|)$
\end{center}
\end{minipage}}
\ar@/^1pc/[uu]_{\substack{q=1\\E\cong v\text{ vector field}\\\ed=\curl,\,\cd=\div}}
\ar@/_3.2pc/[uur]_{\substack{q=0\\E\cong u\text{ function}\\\ed=\grad,\,\cd=0}}
&
&
*+[F-:<0pt>]{\begin{minipage}{4.9cm}
\begin{center}
II. our new inequality\\[1mm]
$|\T|\leq\ch(|\sym\T|+|\Curl\T|)$
\end{center}
\end{minipage}}
\ar@/^2pc/[uul]^{\substack{\T\in\so(3)\\(\T\text{ skew)}}\quad}
\ar@/_1pc/[uu]^{\substack{\T=\Grad v\\(\T\text{ compatible)}}}
}$$
\caption{The three fundamental inequalities are implied by two. 
For the constants we have $\cp=\cpz$, $\cm=\cpo$ and $\ck,\cp\leq\ch$.}
\mylabel{diagram}
\end{figure}

\begin{rem}
\mylabel{fundineq}
Let us consider the fundamental generalized Poincar\'e inequality 
for differential forms, i.e.,
for all $q=0,\dots,3$ there exist constants $\cpq>0$,
such that for all $q$-forms $E\in\Dqcgatom\cap\Deqcganom\cap\harmdiq^{\bot}$
\begin{align}
\mylabel{formineq}
\normLtqom{E}\leq\cpq\big(\normLtqpoom{\ed E}+\normLtqmoom{\cd E}\big)
\end{align}
holds. We note that the analogue of 
Corollary \ref{poincaremaxcor} holds as well.
Here, $E$ is a differential form of rank $q$ 
and $\ed$, $\cd=\pm*\ed*$, $*$ denote the exterior derivative,
co-derivative and Hodge's star operator, respectively.
$\Dqom$ is the Hilbert space of all $\Ltqom$ forms
having weak exterior derivative in $\Ltqpoom$ and 
by $\Dqcgatom$ we denote the closure of smooth forms
vanishing in a neighborhood of $\gat$ with respect to the 
natural graph norm of $\Dqom$. The same construction is used
to define the corresponding Hilbert spaces for the co-derivative $\Deqom$.
Moreover, we introduce $\harmdiq:=\Dqczgatom\cap\Deqczganom$, 
the finite-dimensional space of generalized Dirichlet-Neumann forms.
In classical terms, we have
$$E\in\harmdiq\qequi
\ed E=0,\quad\cd E=0,\quad\iota_{\gat}^{*}E=0,\quad\iota_{\gan}^{*}*E=0,$$
where $\iota_{\gat}:\gat\hookrightarrow\ga\hookrightarrow\om$
denotes the canonical embedding.

Our new inequality, i.e., Theorem \ref{mainlem}, 
together with the generalized Poincar\'e inequality \eqref{formineq} 
imply the three well known fundamental inequalities, i.e.,
\begin{itemize}
\item[1.] the Maxwell inequality, i.e., Lemma \ref{poincaremax},
\item[2.] Poincar\'e's inequality \eqref{poincareestcl},
\item[3.] Korn's inequality, i.e., Lemma \ref{korn} (i).
\end{itemize}
Figure \ref{diagram} illustrates this fact and
Figure \ref{qformtable} shows an identification table 
for $q$-forms and corresponding vector proxies.
\end{rem}

\begin{figure}
\begin{center}
\begin{tabular}{|c||c|c|c|c|}
\hline
$q$ & $0$ & $1$ & $2$ & $3$\\
\hline\hline
$\ed$ & $\grad$ & $\curl$ & $\div$ & $0$\\
\hline
$\cd$ & $0$ & $\div$ & $-\curl$ & $\grad$\\
\hline
$\Dqcgatom$ & $\Hgcgatom$ & $\Hccgatom$ & $\Hdcgatom$ & $\Ltom$\\
\hline
$\Deqcganom$ & $\Ltom$ & $\Hdcganom$ & $\Hccganom$ & $\Hgcganom$\\
\hline
$\iota_{\gat}^{*}E$ & $E|_{\gat}$ & $\nu\times E|_{\gat}$ & $\nu\cdot E|_{\gat}$ & $0$\\
\hline
$\circledast\iota_{\gan}^{*}*E$ & $0$ & $\nu\cdot E|_{\gan}$ & $-\nu\times(\nu\times E)|_{\gan}$ & $E|_{\gan}$\\[0.5mm]
\hline
\end{tabular}
\end{center}
\caption{identification table for $q$-forms and vector proxies in $\rt$}
\mylabel{qformtable}
\end{figure}

\subsection{A generalization: media with structural changes}

Let $\gat\neq\emptyset$ throughout this subsection. 
We consider the case of media with structural changes, see
\cite{Klawonn_Neff_Rheinbach_Vanis09,Neff_micromorphic_rse_05}.
To handle this case we use a result by Neff \cite{Neff00b}, 
later improved by Pompe \cite{pompekorn}, c.f., \eqref{korn_neff}.
To apply the main result from \cite{pompekorn},
let $\TF\in\Czomb$ be a $(3\times3)$-matrix field
satisfying $\det\TF\geq\mu$ with some $\mu>0$. 
Then, there exists a constant $\cksF>0$, 
such that for all $v\in\HGcgatom$
\begin{align}
\mylabel{kornpompe}
\normLtom{\Grad v}\leq\cksF\normLtom{\sym(\Grad v\,\TF)}.
\end{align}
For tensor fields $\T\in\HCom$ we define the semi-norm
$\dnormF{\,\cdot\,}$ by
\begin{align}
\mylabel{hsymnormdefF}
\dnormF{\T}^2:=\normLtom{\sym(\T\TF)}^2+\normLtom{\Curl\T}^2.
\end{align}
Furthermore, there exists a constant $\cF>0$ such that
for all $\T\in\Ltom$
$$\normLtom{\T\TF}\leq\cF\normLtom{\T}.$$

Let us first generalize Lemma \ref{kornlem}.

\begin{lem}
\mylabel{kornlemF}
{\sf(Generalized Korn's First Inequality: Tangential Version)}
There exists a constant $\cktF\geq\cksF$, 
such that the inequality
$$\normLtom{\Grad v}\leq\cktF\normLtom{\sym(\Grad v\TF)}$$
holds for all vector fields
$v\in\HGom$ with $\Grad v\in\HCczgatom$.
\end{lem}

\begin{proof}
The proof is identical with the one of Lemma \ref{kornlem}
using \eqref{kornpompe} instead of Lemma \ref{korn} (i).
\end{proof}

We can generalize Lemma \ref{genkornlem}.

\begin{lem}
\mylabel{genkornlemF}
{\sf(Generalized Korn's Inequality: Irrotational Version)}
There exists a constant $\ckF\geq\cktF$, 
such that the inequality
$$\normLtom{\T}\leq\ckF\normLtom{\sym(\T\TF)}$$
holds for all tensor fields $\T\in\HCczgatom$.
\end{lem}

\begin{proof}
The proof is identical with the one of Lemma \ref{genkornlem} (i)
using Lemma \ref{kornlemF} instead of Lemma \ref{kornlem}.
\end{proof}

Finally, we get:

\begin{theo}
\mylabel{maintheoF}
Let $\chF:=\max\{\sqrt{2}\ckF,\cm\sqrt{1+2\ckF^2\cF^2}\}$.
There exists $c>0$ such that for all $\T\in\HCcgatom$ 
$$\normLtom{\T}\leq\chF\dnormF{\T},\quad
\norm{\T}_{\HCom}\leq c\dnormF{\T}.$$
\end{theo}

\begin{proof}
It is sufficient to prove the first estimate.
Let $\T\in\HCcgatom$. 
Again, we follow in close lines the proof of Theorem \ref{mainlem} (i).
With the same notations 
and using Lemma \ref{genkornlemF} instead of Lemma \ref{genkornlem} (i) we see
\begin{align*}
\normLtom{\T}^2
&=\normLtom{\TR}^2+\normLtom{\TS}^2
\leq\ckF^2\normLtom{\sym(\TR\TF)}^2+\normLtom{\TS}^2\\
&\leq2\ckF^2\normLtom{\sym(\T\TF)}^2+2\ckF^2\normLtom{\sym(\TS\TF)}^2
+\normLtom{\TS}^2\\
&\leq2\ckF^2\normLtom{\sym(\T\TF)}^2+(1+2\ckF^2\cF^2)\normLtom{\TS}^2
\end{align*}
and thus $\normLtom{\T}^2\leq\chF^2\dnormF{\T}^2$.
\end{proof}

\subsection{More generalizations}

Finally we note that there are a lot more generalizations.
In future contributions we will also prove versions of our estimates
\begin{itemize}
\item in $\Lpom$ spaces (possibly just for $p$ near to 2),
\item in unbounded domains, like exterior domains,
\item for domains $\om\subset\rN$ (using differential forms),
\item with inhomogeneous (restricted) tangential traces,
\item concerning the deviatoric part of a tensor.
\end{itemize}

\subsection{Conjectures}

In view of one of the estimates which we have proved 
in this contribution, i.e., Theorem \ref{mainlem} (ii'), 
and the rigidity estimate \eqref{rigidity_friesecke} we speculate that
\begin{align*}
\min_{R\in\SO(3)}\int_{\om}\dist^p(\T,R)\dl
\leq\ck^p\int_{\om}\big(\dist^p(\T,\SO(3))+|\Curl\T|^p\big)\dl
\end{align*}
may hold for some $1<p<\infty$.\footnote{$\ds\int_{\om}\dist^p(\T,\SO(3))\dl$ 
gives an $\Lpom$-control of $\T$ for free, 
contrary to our infinitesimal version, Theorem \ref{mainlem} (ii').}    

A result in Garroni et al. \cite[Th.9]{Garroni10} 
states that for $\om\subset\rz^2$ having Lipschitz boundary\footnote{The
authors assume implicitly that $\om$ is sliceable and probably simply connected.}
there exists $c>0$ such that
\begin{align}
\mylabel{garroni_estimate}
\normLtom{\T}\leq c\big(\normLtom{\sym\T}+|\Curl\T|(\om)\big)
\end{align}
holds for all $\T\in\Loom$ with $\Tskew_{\T}=0$
and $\Tskew_{\T}$ from \eqref{defTpi}.
Here, the term $|\Curl\T|(\om)$ denotes the total variation measure 
of the $\Curl$-operator. 
However, the employed methods are restricted 
to the two-dimensional case since decisive use is made 
of the crucial $\rz^2$-identity $\curl(v_1,v_2)=\div(-v_2,v_1)$, 
see our discussion in \cite{neffpaulywitschgenkornrtzap}.

In view of the inequality \eqref{garroni_estimate} 
we conjecture that for a sliceable
(and maybe simply connected) domain $\om\subset\rN$ 
there exists $c>0$ such that
\begin{align}
\label{conjecture_estimate}
\norm{\T}_{\Lgen{N/(N-1)}{}(\om)}\leq c\big(\normLtom{\sym\T}+|\Curl\T|(\om)\big)
\end{align}
holds for all $\T\in\Loom$ with $\Tskew_{\T}=0$,
where $\Curl\T$ is the natural generalization 
of the $\Curl$-operator to higher dimensions, 
see \cite{neffpaulywitschgenkornrn}. 
This conjecture is based on the observation, 
that for $N=3$ and $\T$ already skew-symmetric 
one cannot be better than 
the well-known Poincar\'e-Wirtinger inequality in $\BVom$, i.e.,
\begin{align*}
\norm{u-\alpha_{u}}_{\Lgen{N/(N-1)}{}(\om)}\leq c|\na u|(\om),\quad
\alpha_{u}:=\pi_{\rz}u=\oint_{\om}u\dl\in\rz.
\end{align*}
The relevance of the latter for \eqref{conjecture_estimate}
is clear by taking into account that
for skew-symmetric matrices $\T$, 
$\Curl\T$ can be interchanged with all partial derivatives, 
see inequality \eqref{trivcurlest}. 
However, new methods have to be developed to tackle this problem.

\appendix

\section{Appendix}

\subsection{Korn's first inequality with full Dirichlet boundary condition}

We note some simple estimates 
concerning the most elementary version 
of Korn's first inequality
for a domain $\om\subset\rN$, $N\in\nz$.
By twofold partial integration we get
$$\scpLtom{\p_{n}v_{m}}{\p_{m}v_{n}}=\scpLtom{\p_{m}v_{m}}{\p_{n}v_{n}}$$
for all smooth vector fileds $v\in\Cicom$ and hence
\begin{align}
\begin{split}
\mylabel{korntrivialcalcone}
\normLtom{\sym\na v}^2
&=\frac{1}{4}\sum_{n,m=1}^N\normLtom{\p_{n}v_{m}+\p_{m}v_{n}}^2\\
&=\frac{1}{2}\sum_{n,m=1}^N\big(\normLtom{\p_{n}v_{m}}^2
+\scpLtom{\p_{n}v_{m}}{\p_{m}v_{n}}\big)\\
&=\frac{1}{2}\big(\normLtom{\na v}^2+\normLtom{\div v}^2\big)
=\frac{1}{2}\normLtom{\curl v}^2+\normLtom{\div v}^2,
\end{split}
\end{align}
which holds for all $v\in\Hocom$ as well.
For a quadaratic matrix $\T$ and $\alpha\in\rz$ 
we define the deviatoric part by
$$\dev_{\alpha}\T:=\T-\alpha\tr\T\id.$$
Then, for any $\alpha\in\rz$ and $v\in\Hoom$
we obtain for the deviatoric part of the symmetric gradient
\begin{align}
\begin{split}
\mylabel{korntrivialcalctwo}
\normLtom{\dev_{\alpha}\sym\na v}^2
&=\normLtom{\sym\na v-\alpha\div v\id}^2\\
&=\normLtom{\sym\na v}^2
+c_{\alpha}\normLtom{\div v}^2,\quad
c_{\alpha}:=\alpha(N\alpha-2).
\end{split}
\end{align}
Combining \eqref{korntrivialcalcone} and \eqref{korntrivialcalctwo} 
we get for any $\alpha\in\rz$ and $v\in\Hocom$
\begin{align*}
\normLtom{\dev_{\alpha}\sym\na v}^2
&=\frac{1}{2}\normLtom{\na v}^2
+(c_{\alpha}+\frac{1}{2})\normLtom{\div v}^2\\
&\geq\frac{1}{2}\normLtom{\na v}^2
+\frac{N-2}{2N}\normLtom{\div v}^2,\\
\normLtom{\dev_{\alpha}\sym\na v}^2
&=\frac{1}{2}\normLtom{\curl v}^2
+(c_{\alpha}+1)\normLtom{\div v}^2\\
&\geq\frac{1}{2}\normLtom{\curl v}^2
+\frac{N-1}{N}\normLtom{\div v}^2,\\
\normLtom{\dev_{\alpha}\sym\na v}^2
&=(2c_{\alpha}+1)\normLtom{\sym\na v}^2
-c_{\alpha}\normLtom{\na v}^2\\
&\geq\frac{N-2}{N}\normLtom{\sym\na v}^2
-c_{\alpha}\normLtom{\na v}^2,\\
\normLtom{\dev_{\alpha}\sym\na v}^2
&=(c_{\alpha}+1)\normLtom{\sym\na v}^2
-\frac{c_{\alpha}}{2}\normLtom{\curl v}^2\\
&\geq\frac{N-1}{N}\normLtom{\sym\na v}^2
-\frac{c_{\alpha}}{2}\normLtom{\curl v}^2.
\end{align*}
We note 
$$
c_{\alpha}=\ct_{\alpha}-\frac{1}{N},\quad
c_{\alpha}+\frac{1}{2}=\ct_{\alpha}+\frac{N-2}{2N},\quad
c_{\alpha}+1=\ct_{\alpha}+\frac{N-1}{N},\quad
\ct_{\alpha}:=N(\alpha-\frac{1}{N})^2.$$
Let us define $I:=(0,\frac{2}{N})$ 
and collect some resulting estimates: 
For all $\alpha\in\rz$ and all $v\in\Hoom$
\begin{align*}
\normLtom{\div v}&\leq\sqrt{N}\normLtom{\na v},\\
\normLtom{\sym\na v}&\leq\normLtom{\na v},\\
\normLtom{\sym\na v}&\leq\normLtom{\dev_{\alpha}\sym\na v}
&&\forall\,\alpha\in\rz\setminus I,\\
\normLtom{\div v}&\leq\frac{1}{\sqrt{c_{\alpha}}}\normLtom{\dev_{\alpha}\sym\na v}
&&\forall\,\alpha\in\rz\setminus\bar{I},\\
\normLtom{\dev_{\alpha}\sym\na v}&\leq\normLtom{\sym\na v}
&&\forall\,\alpha\in\bar{I}.
\intertext{For all $\alpha\in\rz$ and all $v\in\Hocom$}
\normLtom{\curl v},\normLtom{\div v}&\leq\normLtom{\na v},\\
\normLtom{\div v}&\leq\normLtom{\sym\na v},\\
\normLtom{\div v}&\leq\sqrt{\frac{N}{N-1}}\normLtom{\dev_{\alpha}\sym\na v},\\
\normLtom{\na v}&\leq\sqrt{2}\normLtom{\sym\na v},\\
\normLtom{\na v}&\leq\sqrt{2}\normLtom{\dev_{\alpha}\sym\na v},\\
\normLtom{\dev_{\alpha}\sym\na v}&\leq\sqrt{c_\alpha+1}\normLtom{\sym\na v}
&&\forall\,\alpha\in\rz\setminus I,\\
\normLtom{\sym\na v}&\leq\frac{1}{\sqrt{c_\alpha+1}}\normLtom{\dev_{\alpha}\sym\na v}
&&\forall\,\alpha\in\bar{I}.
\end{align*}
We note that the most important case is $\alpha=1/N\in I$, 
where we have $\ct_{\alpha}=0$ and $c_{\alpha}+1=(N-1)/N$ as well as
$$\frac{1}{\sqrt{2}}\normLtom{\na v}\leq\normLtom{\dev_{\frac{1}{N}}\sym\na v}
\leq\normLtom{\sym\na v}\leq\normLtom{\na v}$$
for all $v\in\Hocom$.

\subsection{Korn's first inequality without boundary condition}

By Rellich's selection theorem for $\Hoom$,
Korn's second inequality
and normalization one gets
\begin{align}
\mylabel{kornwithoutbconeapp}
\normLtom{\na v}
&\leq\ck\normLtom{\sym\na v}
\intertext{for all $v\in\Hoom$ with $\na v\bot\so(3)$.
Equivalently, one has for all $v\in\Hoom$} 
\begin{split}
\mylabel{kornwithoutbctwoapp}
\normLtom{(\id-\pi_{\so(3)})\na v}
&\leq\ck\normLtom{\sym\na v},\\
\normLtom{\na v}
&\leq\ck\big(\normLtom{\sym\na v}^2+\normLtom{\pi_{\so(3)}\na v}^2\big)^{1/2}.
\end{split}
\end{align}
Here $\pi_{\so(3)}:\Ltom\to\so(3)$ denotes the $\Ltom$-orthogonal projection 
onto $\so(3)$ and can be expressed explicitly by
$$\pi_{\so(3)}\TP:=\sum_{\ell=1}^{3}\scpLtom{\TP}{\Tskew^{\ell}}\Tskew^{\ell},\quad
\normLtom{\pi_{\so(3)}\TP}^2=\sum_{\ell=1}^{3}|\scpLtom{\TP}{\Tskew^{\ell}}|^2,$$
where $(\Tskew^{\ell})_{\ell=1}^{3}$ is an $\Ltom$-orthonormal basis of $\so(3)$.
Note that $\big(\lambda(\om)^{1/2}\Tskew^{\ell}\big)_{\ell=1}^{3}$ is also an 
$\rttt$-orthonormal basis of $\so(3)$ and thus we have the representation
$$\pi_{\so(3)}\TP
=\sum_{\ell=1}^{3}\scps{\skew\int_{\om}\TP\dl}{\Tskew^{\ell}}_{\rttt}\Tskew^{\ell}
=\skew\oint_{\om}\TP\dl
=:\Tskew_{\TP}\in\so(3).$$
Poincar\'e's inequality for vector fields by normalization reads 
\begin{align}
\mylabel{poincarewithoutbconeapp}
\normLtom{v}
&\leq\cp\normLtom{\na v}
\intertext{for all $v\in\Hoom$ with $v\bot\,\rt$.
Equivalently, one has for all $v\in\Hoom$}
\begin{split}
\mylabel{poincarewithoutbctwoapp}
\normLtom{(\id-\pi_{\rt})v}
&\leq\cp\normLtom{\na v},\\
\normLtom{v}
&\leq\cp\big(\normLtom{\na v}^2+\normLtom{\pi_{\rt}v}^2\big)^{1/2}
\end{split}
\intertext{and hence}
\norm{(\id-\pi_{\rt})v}_{\Hoom}
&\leq(1+\cp^2)^{1/2}\normLtom{\na v},\nonumber\\
\norm{v}_{\Hoom}
&\leq(1+\cp^2)^{1/2}\big(\normLtom{\na v}^2+\normLtom{\pi_{\rt}v}^2\big)^{1/2}.\nonumber
\end{align}
Here $\pi_{\rt}:\Ltom\to\rt$ denotes the $\Ltom$-orthogonal projection 
onto $\rt$ and can be expressed explicitly by
$$\pi_{\rt}v:=\sum_{\ell=1}^{3}\scpLtom{v}{e^{\ell}}e^{\ell},\quad
\normLtom{\pi_{\rt}v}^2=\sum_{\ell=1}^{3}|\scpLtom{v}{e^{\ell}}|^2,$$
where $(e^{\ell})_{\ell=1}^{3}$ is an $\Ltom$-orthonormal basis of $\rt$.
Note that $\big(\lambda(\om)^{1/2}e^{\ell}\big)_{\ell=1}^{3}$ is also an 
$\rt$-orthonormal basis of $\rt$ and thus we have the representation
$$\pi_{\rt}v=\sum_{\ell=1}^{3}\scps{\int_{\om}v\dl}{e^{\ell}}_{\rt}e^{\ell}
=\oint_{\om}v\dl=:a_{v}\in\rt.$$
Combining \eqref{kornwithoutbconeapp} and \eqref{poincarewithoutbconeapp} we obtain
\begin{align}
\mylabel{kornwithpoincarewithoutbconeapp}
(1+\cp^2)^{-1/2}\norm{v}_{\Hoom}\leq\normLtom{\na v}\leq\ck\normLtom{\sym\na v}
\end{align}
for all $v\in\Hoom$ with $\na v\bot\so(3)$ and $v\bot\,\rt$.
Without these conditions one has
\begin{align}
(1+\cp^2)^{-1/2}\norm{(\id-\pi_{\RM})v}_{\Hoom}
&\leq\normLtom{(\id-\pi_{\so(3)})\na v}
\leq\ck\normLtom{\sym\na v},\nonumber\\
(1+\cp^2)^{-1/2}\norm{v}_{\Hoom}
&\leq\big(\normLtom{\na v}^2
+\normLtom{\pi_{\rt}v}^2\big)^{1/2}\mylabel{kornwithpoincarewithoutbctwoapp}\\
&\leq\ck\big(\normLtom{\sym\na v}^2
+\normLtom{\pi_{\so(3)}\na v}^2+\normLtom{\pi_{\rt}v}^2\big)^{1/2}\nonumber\\
&\leq\ck\big(\normLtom{\sym\na v}^2
+\norm{\pi_{\RM}v}_{\Hoom}^2\big)^{1/2},\nonumber
\end{align}
for all $v\in\Hoom$, where $\pi_{\RM}:\Hoom\to\RM$ is defined by
$$\pi_{\RM}v:=(\pi_{\so(3)}\na v)\xi+\pi_{\rt}\big(v-(\pi_{\so(3)}\na v)\xi\big)
=\pi_{\rt}v+(\id-\pi_{\rt})\big((\pi_{\so(3)}\na v)\xi\big)$$
with the identity function $\xi(x):=\id(x)=x$. We note
\begin{align*}
\pi_{\so(3)}\na v&=\Tskew_{\na v}=\skew\oint_{\om}\na v\dl\in\so(3),&
\pi_{\rt}v&=a_{v}=\oint_{\om}v\dl\in\rt,\\
\pi_{\RM}v&=r_{v}:=\Tskew_{\na v}\xi+a_{v}-\Tskew_{\na v}a_{\xi}\in\RM,&
\na\pi_{\RM}v&=\na r_{v}=\Tskew_{\na v}=\pi_{\so(3)}\na v\in\so(3).
\end{align*}
Note that $u:=(\id-\pi_{\RM})v=v-r_{v}$ belongs to $\Hoom$ and satisfies 
$$\na u=(\id-\pi_{\so(3)})\na v\bot\so(3),\quad
u=(\id-\pi_{\rt})\big(v-(\pi_{\so(3)}\na v)\xi\big)\bot\,\rt.$$
Hence \eqref{kornwithpoincarewithoutbconeapp} holds for $u$.
Moreover, we have for $v\in\Hoom$
$$\pi_{\RM}v=0\qequi\pi_{\so(3)}\na v=0\,\wedge\,\pi_{\rt}v=0
\qequi\na v\bot\so(3)\,\wedge\,v\bot\,\rt.$$
This can also be seen be elementary calculations:
For all $\Tskew\in\so(3)$ and all $a\in\rt$ we have
\begin{align}
\mylabel{AvAscp}
\begin{split}
\scpLtom{\Tskew_{\na v}}{\Tskew}
&=\scps{\int_{\om}\na v\dl}{\Tskew}_{\rttt}
=\int_{\om}\scp{\na v}{\Tskew}_{\rttt}\dl
=\scpLtom{\na v}{\Tskew},\\
\scpLtom{r_{v}}{a}
&=\scps{\int_{\om}r_{v}\dl}{a}_{\rt}
=\scps{\Tskew_{\na v}\int_{\om}\xi\dl+\lambda(\om)(a_{v}-\Tskew_{\na v}a_{\xi})}{a}_{\rt}\\
&=\scps{\int_{\om}v\dl}{a}_{\rt}=\scpLtom{v}{a}.
\end{split}
\end{align}
Thus $\na u\bot\so(3)$ and $u\bot\,\rt$.
This shows also that $\na v\bot\so(3)$ and $v\bot\,\rt$ if we have $r_{v}=0$.
On the other hand, if $v\in\Hoom$ with $\na v\bot\so(3)$ and $v\bot\,\rt$,
then $r_{v}=0$ because $\Tskew_{\na v}=0$
by setting $\Tskew:=\Tskew_{\na v}\in\so(3)$
and then $a_{v}=0$ by setting $a:=a_{v}=r_{v}$.

\begin{acknow}
We heartily thank Kostas Pamfilos for his beautiful pictures of sliceable domains.
\end{acknow}

\bibliographystyle{plain} 
\bibliography{/Users/paule/Library/texmf/tex/TeXinput/bibtex/paule,/Users/paule/Library/texmf/tex/TeXinput/bibtex/patrizio}

\end{document}